\documentclass{article}

\usepackage{amsfonts}
\usepackage{amssymb}

\usepackage[usenames,dvipsnames]{color}
\usepackage{amsmath}
\usepackage{amsfonts}
\usepackage{amssymb}
\usepackage{algorithmic,algorithm}
\usepackage{eurosym}
\usepackage{calc}
\usepackage{hyperref}

\usepackage{graphicx}
\usepackage{caption}
\usepackage{subcaption}
\usepackage{placeins}

\usepackage[margin=1in]{geometry}

\usepackage{amsthm}

\usepackage{booktabs}
\usepackage{multirow}
\usepackage{amsthm}
\newtheorem{theorem}{Theorem}
\newtheorem{proposition}{Proposition}
\newtheorem{lemma}{Lemma}

\newtheorem{example}{Example}
\newtheorem{definition}{Definition}
\newtheorem{assumption}{Assumption}
\newtheorem{corollary}{Corollary}
\newtheorem{remark}{Remark}
\newtheorem{observation}{Observation}

\newcommand{\R}{\mathbb{R}}
\newcommand{\X}{\mathbb{X}}
\newcommand{\Z}{\mathbb{Z}}
\newcommand{\U}{\mathbb{U}}
\newcommand{\TTS}{\text{TTS}}

\newcommand{\G}{\mathcal{G}}
\newcommand{\E}{\mathcal{E}}
\newcommand{\V}{\mathcal{V}}
\newcommand{\M}{\mathcal{M}}
\newcommand{\D}{\Omega_d}
\newcommand{\C}{\mathcal{C}}
\newcommand{\Su}{\Omega_s}
\newcommand{\W}{\Omega_w}
\newcommand{\N}{\mathcal{N}}

\newcommand{\rR}{\underline{R}}

\newcommand{\I}{ {\bf I} }

\newcommand{\de}{d_e \big( \rho_e(t) \big)}
\newcommand{\se}{s_e \big( \rho_e(t) \big)}

\title{A Convex Reformulation of the Robust Freeway Network Control Problem with Controlled Merging Junctions}
\author{Marius Schmitt\footnote{Automatic Control Laboratory, ETH Zurich, Physikstrasse 3, 8092 Zurich, Switzerland}, John Lygeros$^*$}

\begin{document}

\maketitle

\begin{abstract}
In the freeway network control (FNC) problem, the operation of a traffic network is optimized using only flow control. For special cases of the FNC problem, in particular the case when all merging flows are controlled, there exist tight convex relaxations of the corresponding optimization problem. However, model uncertainty, in particular regarding the fundamental diagram and predictions of future traffic demand, can be a problem in practice. This uncertainty poses a challenge to control approaches that pursue a model- and optimization-based strategy. In this work, we propose a robust counterpart to the FNC problem, where we introduce uncertainty sets for both the fundamental diagram and future, external traffic demands and seek to optimize the system operation, minimizing the worst-case cost. For a network with controlled merging junctions, and assuming that certain technical conditions on the uncertainty sets are satisfied, we show that the robust counterpart of the FNC problem can be reduced to a convex, finite-dimensional and deterministic optimization problem, whose numerical solution is tractable.
\end{abstract}

\section{Introduction} \label{sec:introduction} 

Dynamics traffic assignment (DTA) refers to a broad spectrum of problems in road traffic control, in which a traffic network is actively controlled via traffic lights, active traffic routing, variable speed limits etc.\ in order to optimize an objective, for example the Total Time Spent (TTS) in the network \cite{peeta2001foundations}. A special case is the freeway network control (FNC) problem, which studies the optimal operation of a road network using only flow control. Such an optimization-based control approach employs a macroscopic traffic model. First-order models, most notably the cell transmission model \cite{daganzo1994cell,daganzo1995cell} derived originally as a discretization of the kinematic wave model \cite{lighthill1955kinematic,richards1956shock}, and second order models such as METANET \cite{messner1990metanet,kotsialos2002traffic} are often used to model traffic dynamics. While potentially inferior in terms of accuracy, first-order models have the advantage that in certain cases, convex optimization problems are obtained, which can be efficiently solved to global optimality. In particular, a linear programming formulation is obtained if a triangular or trapezoidal fundamental diagram is relaxed \cite{ziliaskopoulos2000linear}. Solutions of the relaxed problem do not necessarily satisfy the non-relaxed equations describing the traffic dynamics, but subsequent work has identified special cases in which such a convex relaxation is tight. In particular, this is the case for the freeway ramp metering problem where only onramp and off-ramp junctions are present \cite{gomes2006optimal}, a consequence of monotonicity of the underlying system model \cite{gomes2008behavior}. This result has been generalized to networks with FIFO-diverging junctions and flow-control for all merging junctions, first for the special case of a symmetric, triangular fundamental diagram (i.e.,\ for free-flow velocity equal to congestion-wave speed) \cite{como2016convexity} and later for general, concave, monotone demand and supply functions \cite{schmitt2017exact}. The latter result relies on a monotone reformulation of the system dynamics. Even though it is known that the dynamics of FIFO-diverging junctions are \emph{not} monotone if expressed in the densities \cite{munoz2002bottleneck,coogan2015compartmental}, one can perform a state transformation to obtain a monotone model \cite{schmitt2018monotonicity}. Flow control (or priority control) for merging junctions is necessary to retain monotonicity in the transformed system.

However, a problem for all model- and optimization based control policies is the inherent uncertainty in traffic models. In particular, the models rely on predictions of future traffic demands. Uncertainty in traffic demand predictions is as a major challenge \cite{peeta2001foundations}. In addition, real-world measurements show significant variance in the observed fundamental diagram, in particular during times of congestion, see e.g.\ \cite{dervisoglu2009automatic,kurzhanskiy2010active,de2015grenoble} for recent real-word case studies. The variance in cell capacity, that is, in the maximal equilibrium flow, is particularly impactful \cite{muralidharan2011probabilistic}. While many of the proposed traffic control policies rely purely on feedback to mitigate the effects of uncertainty, but use certainty-equivalent models for prediction and optimization, some recent work considers the effects of uncertainty explicitly. For routing problems, another subclass of DTA, it has been shown that monotone routing policies show favorable resilience to capacity reductions \cite{como2013robust1,como2013robust2} and that such policies can be used to stabilize maximal-throughput equilibria \cite{como2015throughput}. However in the FNC problem (which is the focus of this paper), the dynamics of FIFO-divering junctions pose a challenge. An approach for traffic forecasting that is robust for set-valued uncertainty realizations of external traffic demand and fundamental diagram has been proposed \cite{kurzhanskiy2012guaranteed}. The method relies on monotonicity to keep the problem tractable, however, the non-monotone behavior of FIFO-diverging junctions necessitates the use of over-approximations of the uncertainty sets for the predicted state. This means that the bounds on predicted states can become prohibitively conservative quickly, in particular if the initial state is close to the critical density, the density at the boundary between free-flow and congested states. If one makes the additional assumption that a control policy can always be found, which keeps all roads in free flow, for all possible external traffic demands, then the traffic network dynamics (including FIFO-diverging junctions) are monotone. In turn, one can certify that a policy that keeps the network in free-flow for the worst-case (pointwise-maximal) external demand achieves the same for any other (smaller) external demand pattern \cite{sadraddini2016safety}. However, being able to always keep the traffic network in free-flow is a very strong assumption, which is typically not satisfied for those traffic network that one seeks to control actively. The robustness analysis provided by \cite[Propositions 3 and 4]{como2016convexity} relies on a similar idea, in particular, it does not extend to networks where congestion propagates upstream and blocks upstream diverging junctions. Unfortunately, the non-monotone effects of a congested FIFO-diverging junction are exactly what should be targeted in the FNC problem, in order to realize improvements over the uncontrolled case \cite{coogan2016stability}.

In this work, we are interested in the robust FNC problem, where we explicitly consider uncertainty in future, external traffic demand and in the fundamental diagram in the optimization. In particular, we consider a compartmental model reminiscent of the CTM, with FIFO-diverging junctions, and assume that all flows into merging junctions are controlled. We assume that uncertainty sets of a suitable shape are known and we seek to optimize the worst-case system performance in terms of TTS. A central insight of this work is that for such traffic networks, the problem of optimal control in the presence of uncertainty turns out to be simpler than predicting tight bounds on the evolution of the uncontrolled system. The reason for this seemingly counterintuitive observation is that the worst-case uncertainty realization is hard to identify for the uncontrolled traffic network, in general. For example, an increase in the external traffic demand may lead to an improvement of the overall TTS. However, we show that the same is not true if control of merging flows is available. In this case, the worst-case uncertainty realization can be identified and the optimization becomes tractable. This result is based on monotonicity of a reformulation of the system dynamics, obtained via a state-transformation. In particular, FIFO-diverging junctions are monotone in the transformed dynamics. We also analyze how the optimal feedback policy can be approximated efficiently using receding horizon policies, while retaining robustness guarantees.

The paper is structured as follows: First, we introduce the deterministic traffic network model in Section \ref{sec:problem}. The uncertainty sets are introduced in Section \ref{sec:robust} and subsequently, we define the main problem, the robust FNC problem (\ref{sec:problem_statement}), introduce a monotone reformulation of the system dynamics (\ref{sec:monotone}) and use this reformulation to solve the robust FNC problem (\ref{sec:main_result}). In Section \ref{sec:discussion}, we present examples demonstrating that (partial) control of merging flows is necessary (\ref{sec:counterexamples}), and consider extensions to partially controlled, onramp-merging junctions (\ref{sec:ramp_metering}) and receding-horizon control (\ref{sec:receding_horizon}). The theoretical results are verified numerically in Section \ref{sec:numerical}. We summarize the results in Section \ref{sec:conclusions} and provide suggestions for future work.

\section{Model description} \label{sec:problem} 

We employ a first-order, compartmental model for the traffic network. The model is defined on a directed graph $\G = (\V, \E)$. Edges $e  \in \E \subset \V \times \V$ model parts of the road called \emph{cells}, while vertices $v \in \V$ model junctions. The head of an edge is denoted $\sigma_e$ and the tail $\tau_e$. Traffic flows from tail $\tau_e$ to head $\sigma_e$. We introduce the set of \emph{merging junctions} $\mathcal{M} := \{ v : \deg^-(v) > 1 \} \subset \V$ with $\deg^-(v)$ the in-degree of vertex $v$ and the set of \emph{diverging junctions} $\mathcal{D} := \{ v : \deg^+(v) > 1 \} \subset \V$ with $\deg^+(v)$ the out-degree of vertex $v$. A vertex with $\deg^+(v) = 0$ is called a \emph{sink}. In addition, we define the set of source cells $\mathcal{S} := \{ e : \deg^-( \tau_e ) = 0 \}$. The state of the road network is described by the traffic \emph{density} $\rho_e(t)$ in each cell. We consider a discrete-time model and denote the flow within a time interval of length $\Delta t$ out of cell $e$ as $\phi_e(t)$. For any two edges $e$ and $i$ that are adjacent in the sense that $\tau_e = \sigma_i$, we define split ratios $\beta_{e,i} > 0$ such that $\sum_{i \in \E} \beta_{i,e} \leq 1$. We allow for $\sum_{i \in \E} \beta_{i,e} < 1$ and assume that the remaining percentage of traffic has left the network, for example via an offramp. For ease of notation, we also define $\beta_{e,i} := 0$ for non-adjacent cells ($\tau_e \neq \sigma_i$) and introduce the \emph{routing matrix} $R$ with $R_{e,i} := \beta_{e,i}$. In addition to internal traffic flows, external traffic demand $w_e(t)$ enters the network exclusively via the source cells, that is, $w_e(t) = 0$ for $e \notin \mathcal{S}$.
\begin{assumption}
\label{assumption:graph}
The directed network graph $\G$ does not contain self loops, that is, edges of the form $e = (v,v)$. In addition, $\deg^+(v) = 1$ for all $v \in \M$, that is, merging and diverging junctions are distinct and merging junctions are not sinks. Furthermore, for every edge $j$, there exists a directed path (along edges with nonzero turning ratios) to some edge $e$ with $\sum_{i \in \E} \beta_{i,e} < 1$
\end{assumption} 
\begin{figure} 
	\centering
	\begin{subfigure}[b]{0.45\textwidth}
		\includegraphics[width=\textwidth]{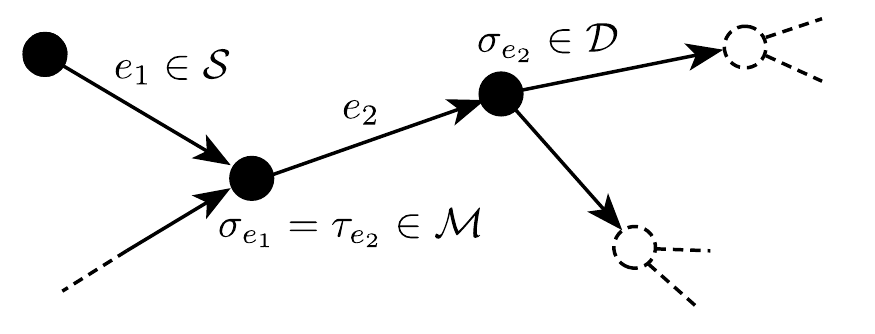}
		\caption{Notation for the network graph. Note that traffic flows from tail $\tau_e$ to head $\sigma_e$.}
		\label{fig:graph_notation}
	\end{subfigure} ~ 
	\begin{subfigure}[b]{0.25\textwidth}
		\includegraphics[width=\textwidth]{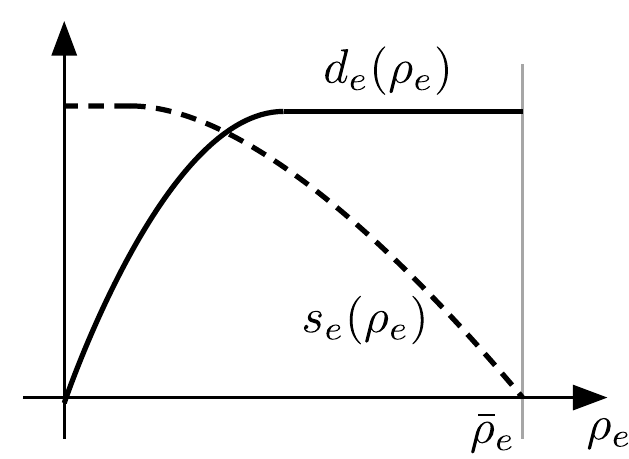}
		\caption{Demand and supply function satisfying Assumption \ref{assumption:fd} for suitable $\Delta t$.}	
		\label{fig:fd_concave}
	\end{subfigure} 
	\caption{The conservation law of the compartmental model is based on the structure of the network graph, while the uncontrolled flows are determined by traffic demand and supply.}
\label{fig:fd}
\end{figure} 
The latter condition implies that as long as traffic keeps moving, all traffic eventually leaves the network \cite{varaiya2013max,coogan2016stability}. In particular, it implies that the spectral radius of the routing matrix is strictly less than one. 
The evolution of the traffic densities $\rho_e(t)$ is described by the conservation law
\begin{equation}
\label{eq:ctm_densities}
\rho_e(t+1) = \rho_e(t) + \frac{\Delta t}{l_e} \cdot \left( \sum_{i \in \E} \beta_{e,i} \phi_i(t) ~ - \phi_e(t) + w_e(t) \right) \quad \forall e \in \E .
\end{equation}
To complete the model, we need to define the flows as a function of the densities. In the CTM, traffic \emph{demand} $\de$ and \emph{supply} $\se$ are introduced for each cell, to model the amount of cars that seek to travel downstream or the amount of free space in a cell, respectively. In the following, we assume that:
\begin{assumption} 
\label{assumption:fd}
For every cell $e$, we introduce a traffic jam density $\bar\rho_e$. The demand function $d_e( \rho_e(t) )$, $d_e: [0,\bar \rho_e] \rightarrow \mathbb{R}_+$ is concave, Lipschitz-continuous with Lipschitz constant $\gamma$, nondecreasing and it satisfies $d_e(0) = 0$. Conversely, the supply function $s_e( \rho_e(t) )$, $s_e: [0,\bar \rho_e] \rightarrow \mathbb{R}_+$ is concave, Lipschitz-continuous with Lipschitz constant $\gamma$, nonincreasing and it satisfies $s_e(\bar \rho_e) = 0$. Furthermore, the sampling time $\Delta t$ is chosen such that it satisfies the bounds 
\begin{align}
\label{eq:lipschitz}
\Delta t \leq \frac{l_e}{ \gamma } \quad \forall e \in \E .
\end{align}
We allow for cells with infinite capacity, that is, cells with both infinite traffic jam densities $\bar\rho_e = +\infty$ and infinite supply $s_e( \rho_e(t) ) = +\infty$ for all $\rho_e(t) \in [0, +\infty)$. In such a case, the demand function $d_e( \rho_e(t) )$ has to be defined for $\rho_e(t) \in [0, +\infty)$.
\end{assumption} 
Demand and supply function together are called the fundamental diagram of a cell. In this work, we assume that merging flows are controlled, whereas non-merging flows are determined by the fundamental diagrams of the corresponding cells and by the First-In, First-Out (FIFO) rule for diverging junctions. Let $\N := \{e : \sigma_e \in \M \}$ denote the set of indices of flows into a merging junction. All other flows are given as
\begin{equation}
\label{eq:ctm_flows}
\phi_e(t) =  \min \left\{ \de, \underset{i \in \E^+(e)}{\min} \left\{ \frac{ 1}{ \beta_{e,i} }s_i \big( \rho_i(t) \big) \right\} \right\}  \quad \forall e \in \E \setminus \N ~.
\end{equation}
Different models exist for merging junctions, most notably Daganzo's priority rule \cite{daganzo1995cell} and the proportional-priority merging model \cite{kurzhanskiy2010active,coogan2016stability}. We make the additional assumption that merging junctions are controlled \cite{como2016convexity,schmitt2017exact}: for all cells $e \in \N$, the flows $\phi_e(t)$ are chosen by a control policy, subject to demand constraints $0 \leq \phi_e(t) \leq \de$ and supply constraints $\sum_{e \in \E} \beta_{e,i} \cdot \phi_e(t) \leq s_i \big( \rho_i(t) \big)$, where $i$ is the index of the unique cell immediately downstream of the merging junction. This model is consistent with both Daganzo's priority rule and the proportional-priority merging model.

Note that according to the conservation law \eqref{eq:ctm_densities}, external inflows $w_e(t)$ into source cells are not a priori constrained by the supply of free space. To ensure that the system evolution is well defined, we assume that all source cells have infinite capacity $\bar \rho_e = + \infty$ for all $e \in \mathcal{S}$. Alternative models, where surplus external demand is disregarded, have also been proposed. However, these models are not suitable for the objective of minimizing the TTS, as they create an incentive to create ``artificial", temporary congestion with the objective to prevent parts of the external demand from entering the network, see the discussion in \cite{schmitt2017exact}. However, we allow for constraints 
\begin{equation}
\label{eq:ramp_constraints}
\rho_e(t) \leq \bar s_e \quad \forall e \in \mathcal{S}
\end{equation}
to be imposed in the FNC optimization problem. Here $\bar s_e$ is a design specification for the optimization problem, limiting the maximal traffic density in source cells of the \emph{controlled} traffic network.\footnote{Restrictive choice of $\bar s_e$ can lead to infeasibility of the optimization problem \eqref{eq:fnc}, of course.} Collecting the models for the individual components, we arrive at the deterministic traffic network model.

\begin{definition} 
\label{definition:ctm}
Consider a graph $\G = (V,\E)$ satisfying Assumption \ref{assumption:graph}, where each edge is equipped with a fundamental diagram satisfying Assumption \ref{assumption:fd}. We define the \emph{CTM with controlled merging junctions} as the system with states $\rho_e(t)$ for $e \in \E$ and inputs $\phi_e(t)$ for $e \in \N$. The state evolves according to the conservation law \eqref{eq:ctm_densities} with the uncontrolled flows determined by the fundamental diagram \eqref{eq:ctm_flows}. The controlled flows are subject to the constraints
\begin{subequations}
\begin{align}
0 \leq \phi_e(t) &\leq \de && \forall e \in \N , \label{eq:ctm_demand} \\ 
\sum_{i \in \E} \beta_{i,e} \cdot \phi_i(t) &\leq \se && \forall e : \sigma_e \in \M . \label{eq:ctm_supply}
\end{align}
\label{eq:ctm_constraints}
\end{subequations}
\end{definition} 
The CTM with controlled merging junctions according to Definition \ref{definition:ctm} is a special case of the model studied in \cite{schmitt2017exact}, which considers additional variants for merging junctions. In this model, the set $\mathbb{P} := \prod_{e \in \E} [0, \bar\rho_{e}]$ , with $\bar\rho_e = + \infty$ for cells of infinite capacity, is invariant for every feasible input $\phi_e(t)$ ($e \in \N$), and for every $\rho(t) \in \mathbb{P}$, there exists a feasible input. Hence, the system evolution is well defined and we can proceed to the task of optimal control of such a network. A natural objective in traffic control is to minimize the TTS on the road
\begin{align*}
\TTS = \Delta t \cdot \sum_{t=0}^{T} \sum_{e \in \E} l_e \cdot \rho_e(t)
\end{align*}
over a time horizon $\{0,1,\dots,T\}$. Whereas many control problems are defined over an infinite horizon, it is reasonable to assume that traffic networks periodically reach states of low traffic densities, at least every night. It is not advantageous to extend the optimization over such periods and therefore, the horizon of interest will always be finite. Assuming that all model parameters are known, we can formulate the deterministic FNC problem with controlled merging junctions
\begin{equation} 
\label{eq:fnc}
\begin{array}{rrl}
& \underset{\phi(t), \rho(t)}{\text{minimize}} & \TTS  \\ [2ex]
 & \text{subject to} & \text{CTM dynamics \eqref{eq:ctm_densities}, \eqref{eq:ctm_flows} and constraints \eqref{eq:ramp_constraints}, \eqref{eq:ctm_constraints}} \\[0.5ex]
 & & \text{$\rho(0)$ given.}
\end{array}
\end{equation}
The efficient solution of this problem is studied in \cite{schmitt2017exact}. In this work, we extend the scope and consider the influence of model uncertainty and uncertain predictions of future, external traffic demand.

\section{Robust Counterpart} \label{sec:robust} 

In this work, we are interested in the optimal control of traffic networks with controlled merging flows in the presence of uncertainty. In particular, we consider uncertainty in future traffic demand and uncertainty in the fundamental diagram. Future traffic demands are typically predicted using historical data. The system dynamics, that is, the fundamental diagram, can also be estimated from measurements. While measurements of free-flow speed typically show little variance, data collected during times of congestion are often scattered and cannot be well approximated by a one-dimensional fundamental diagram \cite{dervisoglu2009automatic,kurzhanskiy2010active}, as depicted in Figure \ref{fig:fd_uncertain_data}. The variance in cell capacity, that is, in the maximal equilibrium flow, is particularly impactful \cite{muralidharan2011probabilistic}, which has been emphasized as an argument to prefer feedback control based on density measurements \cite{papageorgiou1991alinea}, instead of feedforward-planning based on balancing flows and (uncertain) cell capacities. 

In addition, turning rates and density estimates are uncertain in practice. In this work, we assume that turning rates are known and constant. This is a concession to the FNC problem, where turning rates encode route choices. The FNC problem only allows for flow control, but not for active traffic routing. Time-varying (or uncertain) turning rates could be used to re-route traffic, by using flow control to release traffic at times of ``favorable" turning rates \cite{como2016convexity,schmitt2017exact}. Uncertainty in the densities, that is, in the system state, leads to a combined problem of state estimation and optimal control, for a nonlinear system for which no separation principle holds. Such a problem is out of scope of this work, and we focus on the simpler problem with uncertain dynamics, but perfect state knowledge.

\begin{figure} 
	\centering
	\begin{subfigure}[b]{0.3\textwidth}
		\includegraphics[width = \textwidth]{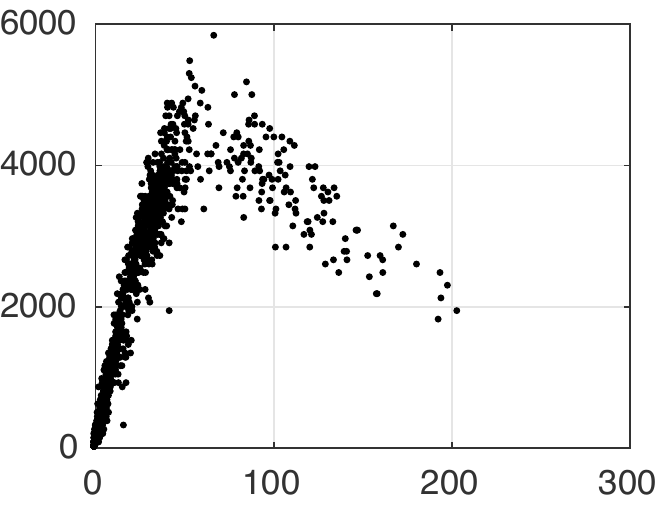}
		\caption{Traffic data from the Rocade Sud, April 14$^{\text{th}}$, 2014.}
		\label{fig:fd_uncertain_data}
	\end{subfigure} ~ 
		\begin{subfigure}[b]{0.3\textwidth}
		\includegraphics[width = \textwidth]{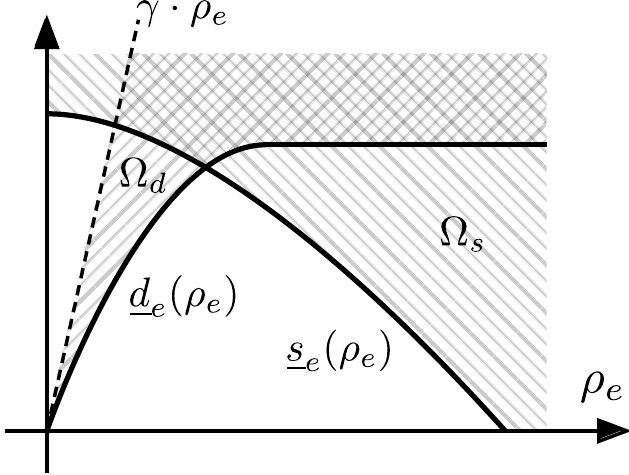}
		\caption{Uncertainty sets for of demand and supply functions.}
		\label{fig:fd_uncertain}
	\end{subfigure} ~ 
		\begin{subfigure}[b]{0.3\textwidth}
		\includegraphics[width = \textwidth]{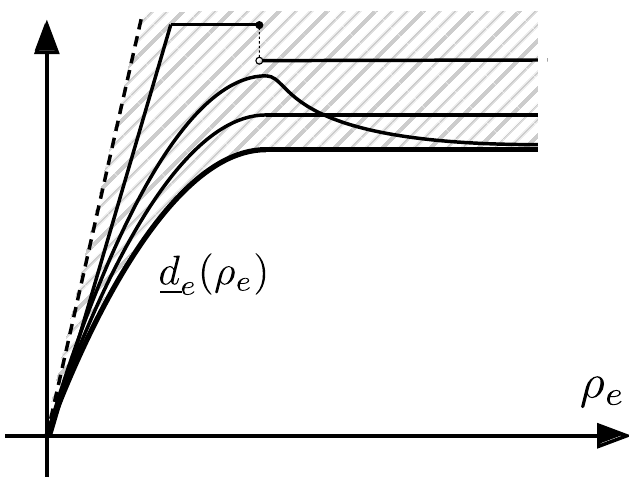}
		\caption{Examples of demand functions contained in the uncertainty set.}
		\label{fig:fd_uncertain_ex}
	\end{subfigure} ~ 
\caption{Real traffic data, such as the data from the Rocade Sud, a freeway around Grenoble, France \cite{de2015grenoble}, show significant variance in the fundamental diagram, in particular in the congested region. Note that only the minimal demand function $\underline d_e( \cdot )$ has to satisfy Assumption \ref{assumption:fd}. Hence, also demand functions with capacity drop (a decrease in demand above a critical density \cite{kontorinaki2017first}) are contained in the uncertainty set $\D$.}
	\label{fig:uncertainty}
\end{figure} 
Following the robust control paradigm, we do not assume knowledge of a probability distribution of the uncertain quantities, but introduce uncertainty sets instead. In particular, we assume that the realizations of the external demand $w_e(t)$ are nonnegative and point-wise upper bounded, $0 \leq w_e(t) \leq \overline w_e(t)$. The corresponding uncertainty sets are defined as $\Omega_{w,t,e} := \{ w_e(t) : 0 \leq w_e(t) \leq \overline w_e(t) \}$. We also use $\Omega_{w,t}$ (respectively $\Omega_{w}$) to denote the joint uncertainty set for all $e \in \E$ (and all $t \in \{0,t,\dots,T\}$).

Similarly, we assume that \emph{lower} bounds on demand $\underline d_e ( \cdot )$ and supply functions $\underline s_e ( \cdot )$ are known for each cell. These minimal demand $\underline d_e ( \cdot )$ and supply functions $\underline s_e ( \cdot )$ are assumed to satisfy Assumption \ref{assumption:fd}, for every $e \in \E$. The real demand $d_{t,e} : \R_+ \to \R_+$ and supply functions $s_{t,e} : \R_+ \to \R_+$ are unknown a priori, and in particular, they do not necessarily satisfy Assumption \ref{assumption:fd}. Here, $\R_+ = [0, +\infty)$ denotes the non-negative numbers and the subscript $t$ indicates that realizations of demand and supply functions may be time-varying.\footnote{In principle, the minimal demand and supply functions can also be time-varying, which might be useful to describe predictable changes in environmental conditions, e.g.\ due to weather conditions. However, we assume time-invariant bounds for ease of exposition.} However, we do assume that they are point-wise lower bounded, that is, $d_{t,e}( \rho ) \geq \underline d_e ( \rho )$, for all $\rho \in 0 \leq \bar \rho_e$, and $d_{t,e}( \rho ) \geq 0$ for $\rho \geq \bar \rho_e$. Likewise $s_{t,e}( \rho ) \geq \underline s_e ( \rho )$, for all $\rho \in [0, \bar \rho_e]$, and $s_{t,e}( \rho ) \geq 0$ for $\rho \geq \bar \rho_e$. In addition, we assume that $d_{t,e}( \rho ) \leq \gamma \cdot \rho$, where $\gamma$ is the Lipschitz constant from Assumption \ref{assumption:fd}. The latter condition ensures that $\Delta t \cdot \phi_e(t) \leq l_e \rho_e(t)$ and hence, that the set $\R_+^n$ (with $n$ the number of cells) is invariant. Note that since the uncertain demand and supply functions are defined for all $\rho_e(t) \in \R_+$, the system evolution remains well defined. However, the densities $\rho_e(t)$ encountered when simulating the uncertain system may exceed $\bar \rho_e$, the traffic jam density of the minimal supply function $\underline s_e(\cdot)$. We define uncertainty sets for demand and supply functions
\begin{align*}
\Omega_{d,e} &:= \{ d_{t,e}: \R_+ \to \R_+ : \underline d_e ( \rho ) \leq d_{t,e}( \rho ),  ~\forall \rho \in [0, \bar \rho_e],~ d_{t,e}( \rho ) \leq \gamma \cdot \rho \}, \\
\Omega_{s,e} &:= \{ s_{t,e}: \R_+ \to \R_+: \underline s_e( \rho ) \leq s_{t,e}( \rho ) , ~\forall \rho \in [0, \bar \rho_e] \}.
\end{align*}
Again, $\Omega_{d}$ and $\Omega_{s}$ denote the joint uncertainty sets for all $e \in \E$. In addition, we introduce the joint uncertainty set $\Omega = \W \times \D^T \times \Su^T$ for ease of notation. Elements $\omega \in \Omega$ are tuples containing $w_e(t)$, $d_{t,e}(\cdot)$ and $s_{t,e}(\cdot)$ for all $e \in \E$ and $t \in \{0,t,\dots,T\}$. In particular, we define $\overline \omega$ as the tuple containing the maximal external demand $\overline w_e(t)$ and the \emph{minimal} demand and supply functions $\underline d_{e}(\cdot)$ and $\underline s_{e}(\cdot)$ for all $e \in \E$ and all $t \in \{0,t,\dots,T\}$. Examples of uncertainty sets for demand and supply functions are depicted in Figure \ref{fig:uncertainty}.
\begin{assumption} 
\label{assumption:uncertainty}
The uncertainty sets for external demand $\W$, demand functions $\D$ and supply functions $\Su$ are of the form described in the previous paragraph and they are known a priori.
\end{assumption} 

\begin{remark}
Depending on the application, it might be desirable to define uncertainty sets such that a ``maximal" traffic jam density $\bar{\bar \rho}_e$, which is larger or equal to the ``minimal" traffic jam density $\bar \rho_e$ of the lower bound on the supply function, is not exceeded. This can be achieved by restricting the uncertainty set for the supply functions further, in particular, by defining $\Omega_{s,e} := \{ s_{t,e}: \R_+ \to \R_+: \underline s_{e}( \rho_e ) \leq s_{t,e}( \rho_e ) ~\forall \rho_e \in [0, \bar \rho_e],~ s_{t,e}(\rho_e) \leq \gamma (\bar{\bar \rho}_e - \rho_e) , ~\forall \rho \in [0, \bar {\bar \rho}_e] \}$. The theoretical results of this work do not depend on the existence or non-existence of such an upper bound on the supply functions. 
\end{remark}

A key insight of this work is that the problem of optimal control of uncertain networks with controlled merging flows turns out to be simpler than the problem of predicting tight bounds on the evolution of the uncontrolled system. The reason for this seemingly counterintuitive observation is that the worst-case uncertainty realization is hard to identify for the uncontrolled system, in general. One can construct examples where, for example, an increase in the external traffic demands leads to an overall reduction in TTS (see Example \ref{example:ex1}, Section \ref{sec:counterexamples}). However, the same is not true if control of merging flows is available, and in the remainder of this section, we seek to demonstrate that in such a case, the worst-case uncertainty realization is easy to identity, making the robust, optimal control problem tractable.

\subsection{Problem statement} \label{sec:problem_statement} 

In this section, we define the robust counterpart to the FNC problem formally. In control, it is typical that uncertainty realizations are revealed stage wise, i.e., in the uncertain FNC problem, the realizations of $d_{t,e}(\cdot)$, $s_{t,e}(\cdot)$ and $w_e(t)$ are revealed at time $t$ (or can be computed from the states observed at $t+1$). The robust control problem aims to mitigate the effects of future uncertainty, while taking past realizations into account. In general, such a control problem requires the optimization over policies, i.e.\ functions $\pi_t: \X \to \U$, where $\X$ is the state space and $\U$ is the input space, mapping states to control inputs. Note that since the state at time $t$ depends on earlier uncertainty realizations, control actions at time $t$ may implicitly depend on uncertainty realizations at earlier time steps. Here, we consider problems over a finite horizon $t \in \{ 0,1,\dots,T \} $ subject to time-varying external disturbances. Therefore, we need to consider time-varying policies as well, which do not only depend on the traffic network state $\rho(t)$, but also on the time $t$. Recall that the control actions in the FNC problem, that is, the controlled flows, are subject to joint state-input constraints. Therefore, we define the set of \emph{feasible} control policies 
\begin{align*}
\Pi_t := \left\{\pi_t : \R_+^{|E|} \to \R_+^{|\N |} : \begin{array}{l} \text{ For all $e \in \N$ and } \phi_e(t) = \pi_{t,e} (\rho(t)) \text{, the pair $ \big( \rho(t), \phi_e(t) \big)$ satisfies} \\[0.5ex] \text{the constraints \eqref{eq:ramp_constraints} and \eqref{eq:ctm_constraints}, for any uncertainty realization $\omega_t \in \Omega_t$.} \end{array} \right\} .
\end{align*}
The joint set $\Pi$ contains the sets of feasible policies $\Pi_t$ for all $t \in \{0,1,\dots,T\}$. It is not a priori clear that this set is non-empty, but in later sections, we will provide conditions which ensure just that. For any particular uncertainty realization $\omega \in \Omega$ and any particular choice of a feasible policy $\pi \in \Pi$, the evolution of the traffic network is well-defined. We denote the resulting cost over the horizon $t \in \{0,1,\dots,T\}$ as $\TTS( \pi, \omega)$. The \emph{robust counterpart} to the FNC problem, or simply the \emph{robust FNC} problem, is defined as the optimization problem
\begin{equation} 
\label{eq:robust_counterpart}
\C^* = \min_{\pi \in \Pi} \left( \max_{\omega \in \Omega} ~ \TTS( \pi, \omega) \right)
\end{equation}
with optimal value (cost) $\C^*$. The robust FNC problem in this form is a bi-level, infinite dimensional and non-convex optimization problem. We aim to solve the robust FNC problem, in the sense that we will provide a finite-dimensional, convex reformulation. The result is based on monotonicity arguments and in the following, we will define and analyze a reformulation of the network model, obtained via a state-transformation, which exhibits favorable monotonicity properties.

\subsection{Monotone reformulation} \label{sec:monotone} 

\begin{figure} 
	\centering
	\begin{subfigure}[b]{0.43\textwidth}
		\includegraphics[width = \textwidth]{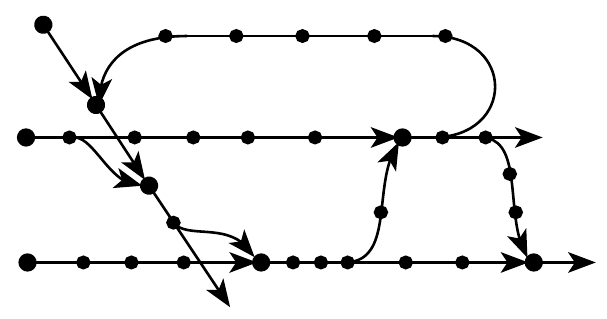}
		\caption{Example network graph, corresponding to a routing matrix $R$ with merging junctions.}
		\label{fig:graph_P}
	\end{subfigure} ~ 
	\begin{subfigure}[b]{0.49\textwidth}
		\includegraphics[width = \textwidth]{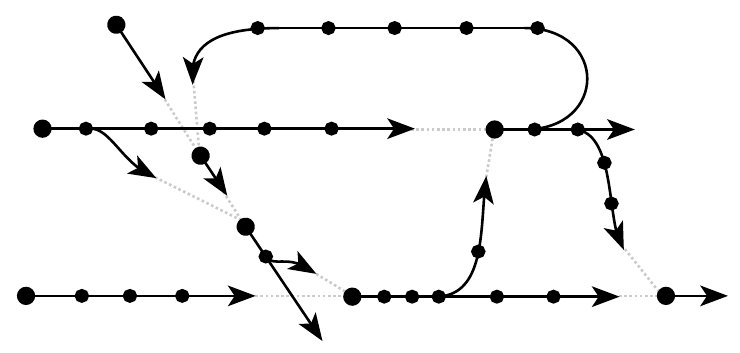}
		\caption{Reduced network graph, corresponding to the reduced routing matrix $\underline{R}$, without merging junctions.}
		\label{fig:graph_Pz}
	\end{subfigure} 
\caption{The state transformation is based on the routing matrix of the reduced network graph, which results from removing merging junctions from the original network graph. A network defined on the graph depicted here is studied in detail in the numerical evaluation in Section \ref{sec:numerical}.
 }
	\label{fig:graph}
\end{figure}

Consider the reduced routing matrix $\rR$, which is defined as $\rR_{i,j} = 0$ if $j \in \N$ and $\rR_{i,j} = R_{i,j} = \beta_{i,j}$ otherwise. The reduced routing matrix describes a graph resembling the original network graph defined by $R$, but where all merging edges are instead converted into sinks and cells downstream of merging junctions are now sources, as depicted in Figure \ref{fig:graph}. Note that the reduced network graph is a forest, that is, a graph whose connected components are (directed) trees. 
Directed trees satisfy Assumption \ref{assumption:graph} and hence, the spectral radius of the reduced routing matrix is also strictly less than one. Therefore, $( \I - \rR)$ is regular and the inverse $( \I - \rR)^{-1}$ exists. In addition, the inverse is nonnegative, since $\rR \geq 0$ and $( \I - \rR )^{-1} = \sum_{k=0}^{\infty} \rR^k \geq 0$. For ease of notation, we define $P := ( \I - \rR)^{-1}$ and $p_e^\top := P_{(e,:)}$ as the row of $P$ corresponding to cell $e$. Based on the reduced routing matrix, we consider the state transformation 
\begin{equation}
\label{eq:transformation}
z(t) := P L \rho(t).
\end{equation}
Within the reduced network, the state $z_e(t)$ can be interpreted as the \emph{backlog} of traffic, that is, the total traffic volume currently in the (reduced) network, that will pass through cell $e$ in the future \cite{schmitt2018monotonicity}. 
In addition, we perform a re-parametrization of the system inputs and introduce generic new inputs $v_e(t)$ for all $e \in \N$. The corresponding, controlled flows are chosen as $\phi_e(t) := \max \left\{ 0, \frac{z_e(t) - v_e(t)}{\Delta t} \right\}$. Note that the non-negativity constraints for flows are implicitly satisfied, but any value $\phi_e(t) \geq 0$ can be obtained for an appropriate choice of $v_e(t)$. 
\begin{lemma} \label{lemma:equivalence_tctm} 
The CTM with controlled merging junctions can equivalently be described by the system with state $z(t)$ and input $v_e(t)$ for $e \in \N_S$, that evolves as
\begin{align}
\label{eq:tctm_dynamics}
z(t+1) = f \big( z(t), v(t), \omega_t \big) := z(t) - \Delta t \cdot \phi(t) + \Delta t \cdot P \big( w(t) + ( R - \rR ) \cdot \phi(t) \big) ,
\end{align}
for uncertainty realization $\omega_t$. The flows are computed as 
\begin{align*}
\phi_e(t) &= \min \left\{ d_{t,e} \big( \rho_e(z(t)) \big), \underset{i \in \E^+(e)}{\min} \left\{ \frac{ s_{t,i}( \rho_i( z(t) ) )}{ \beta_{e,i} } \right\} \right\} & \forall e \notin \N \\[0.5ex]
\phi_e(t) &= \max \left\{ 0, \frac{z_e(t) - v_e(t)}{\Delta t} \right\} & \forall e \in \N .
\end{align*}
In the latter equations, $\rho_e ( z(t) )$ means that $\rho_e(t)$ is computed from the transformed state $z(t)$ via the inverse transformation $\rho(t) = L^{-1} (\I - \rR) z(t)$. We call the transformed system the \emph{transformed cell transmission model} (TCTM). The system is defined on the state-space $\Z := \{ z(t) : L^{-1} (\I - \rR) z(t) \in \mathbb{P} \}$, for inputs $v_e(t)$, $e \in \N$, subject to the constraints
\begin{subequations}
\label{eq:tctm_constraints}
\begin{align}
z_e(t) - v_e(t) - \Delta t \cdot d_{t,e}\big( \rho_e(z(t)) \big) &\leq 0 & \forall e : \sigma_e \in \M, \\[0.5ex]
\sum_{i \in \E} \beta_{e,i} \cdot \max \bigg\{ 0, \frac{z_i(t) - v_i(t)}{\Delta t} \bigg\} - s_{t,e} \big( \rho_e( z(t) ) \big) &\leq 0 & \forall e : \tau_e \in \M, \\[-0.5ex]
z_e(t) &\leq \bar s_e &\forall e \in \mathcal{S}.
\end{align}
\end{subequations}
We introduce the short-hand notation $g \big( z(t), v(t), \omega_t \big) \leq 0$ (where the inequality is interpreted element-wise) for all of these constraints.
\end{lemma} 
\begin{proof} 
The transformation \eqref{eq:transformation} is invertible, and we can verify that applying the transformation yields the equations of the TCTM. The evolution of the transformed system is given as
\begin{align*}
z(t+1) &= (\I - \rR)^{-1} L \cdot \Big( \rho(t) + \Delta t \cdot L^{-1} \big( (R - \I) \phi(t) + w(t) \big) \Big) \\
 &= z(t) + \Delta t \cdot P \big( w(t) + ( R - \rR ) \cdot \phi(t) \big) - \Delta t \cdot \phi(t) .
\end{align*}
The flow equations for uncontrolled flows follows from substituting $\rho_e ( z(t) )$ for $\rho_e(t)$ in the flow equations \eqref{eq:ctm_flows}, while the controlled flows are chosen according to the definition of the inputs $v(t)$. The constraints of the TCTM follow from substituting the re-parametrized inputs into the original constraints \eqref{eq:ramp_constraints} and \eqref{eq:ctm_constraints}. Note that we have used the fact that $\rho_e(t) = z_e(t)$ for all source cells $e \in \mathcal{S}$.
\end{proof} 

The main motivation to introduce the TCTM is to employ monotonicity for its analysis. The standard concept of a monotone function can be generalized to dynamic systems in the following manner.
\begin{definition} \label{definition:monotone_system} 
Consider a dynamical system with state $x(t) \in \X$ and input $u(t) \in \U$, defined by its (potentially time-varying) system dynamics $f_t: \X \times \U \to \X$, i.e.\ $x(t+1) = f_t \big( x(t), u(t) \big)$, equipped with joint input-state constraints $g_t \big( x(t), u(t) \big) \leq 0$. The system is monotone (equipped with monotone constraints) if the functions $f_t( \cdot)$ (the functions $g_t(\cdot )$) are monotone in state $x(t)$ and input $u(t)$.
\end{definition}
Monotone systems with inputs are a generalization of monotone maps \cite{hirsch2005monotone}. An overview of results for monotone systems is provided in e.g.\ \cite{angeli2003monotone,hirsch2006monotone}, although for the continuous-time case. Together with convexity of system dynamics, monotonicity can be used to make certain optimal control problems tractable \cite{rantzer2014control,schmitt2017exact}.

The following technical Lemma will be helpful in analyzing monotonicity of the TCTM.
\begin{lemma} 
\label{lemma:auxiliary}
The auxiliary functions $f^d_e \big( z(t) \big) := z_e(t) - \Delta t \cdot \underline d_e \big( \rho_e(z(t)) \big)$ for all $e \in \E$, $f^s_{e,i} \big( z(t) \big) := z_i(t) - \frac{\Delta t}{\beta_{e,i}} \cdot \underline s_e \big( \rho_e(z(t)) \big)$ for all $e \in \E$, $i \notin \N$ and $\beta_{e,i} > 0$, and $f^r \big( z(t), v(t) \big)  := ( R - \rR ) \cdot \phi(t)$ are monotone in $z(t) \in \Z$.
\end{lemma} 
\begin{proof} 
In the following, we drop the time index $t$ for ease of notation and introduce $\Delta z \geq 0$. To prove monotonicity of the first two auxiliary functions, we resort to the elementary definition of monotonicity. A function $f(z)$ is monotone if $f( z+\Delta z ) \geq f(z)$ for all $\Delta z \geq 0$ such that $z \in \Z$ and $z + \Delta z \in \Z$. We find that 
\begin{align*}
f^d_e \big( z + \Delta z \big) - f^d_e \big( z \big) &= \Delta z_e - \Delta t \cdot \left( \underline d_e \bigg( \frac{z_e + \Delta z_e - \sum_{i \notin \N} \beta_{e,i} z_i + \Delta z_i}{l_e} \bigg) - \underline d_e \bigg( \frac{z_e - \sum_{i \notin N} \beta_{e,i} z_i }{l_e} \bigg) \right) \\[1ex]
&\geq \Delta z_e - \Delta t \cdot \left( \underline d_e \bigg( \frac{z_e + \Delta z_e - \sum_{i \notin \N} \beta_{e,i} z_i }{l_e} \bigg) - \underline d_e \bigg( \frac{z_e - \sum_{i \notin \N} \beta_{e,i} z_i }{l_e} \bigg) \right) \\[1ex]
 &\geq \Delta z_e - \Delta t \cdot \gamma \cdot l_e^{-1} \cdot \Delta z_e \geq 0 ,
\end{align*}
which implies that the auxiliary function $f^d_e \big( \cdot )$ is monotone. Similarly, we find that 
\begin{align*}
f^s_{e,i} \big( z + \Delta z \big) - f^s_{e,i} \big( z \big) &= \Delta z_i - \frac{\Delta t}{\beta_{e,i}} \cdot \left( \underline s_e \bigg( \frac{z_e + \Delta z_e - \beta_{e,i} z_i + \Delta z_i}{l_i} \bigg) - \underline s_e \bigg( \frac{z_e - \beta_{e,i} z_i }{l_i} \bigg) \right) \\[1ex]
 &\geq \Delta z_i - \frac{\Delta t}{\beta_{e,i}} \cdot \left( \underline s_e \bigg( \frac{z_e - \beta_{e,i} z_i + \Delta z_i}{l_i} \bigg) - \underline s_e \bigg( \frac{z_e - \beta_{e,i} z_i }{l_i} \bigg) \right) \\[1ex]
 &\geq \Delta z_i - \frac{\Delta t}{\beta_{e,i}} \cdot \gamma \cdot \frac{\beta_{e,i}}{l_i} \Delta_i \geq 0 ,
\end{align*}
which implies that the auxiliary function $f^s_{e,i}( \cdot )$ is monotone as well. Finally, we consider the auxiliary function $f^r \big( z(t), v(t) \big)  = ( R - \rR ) \cdot \phi(t)$. Each component can be written as $f^r_e \big( z(t), v(t) \big) = \sum_{i \in \N_S} \beta_{e,i} \phi_i(t)$ Since the flows $\phi_i(t)$ for all $i \in \N_S$ are controlled, we have that $f^r_e \big( z(t), v(t) \big) = \sum_{i \in \N_S} \beta_{e,i} \max \big\{ 0, \frac{z_i(t) - v_i(t)}{\Delta t} \big\}$, which is clearly monotone in $z(t)$.
\end{proof} 

To analyze the influence of uncertainty, we will now fix the inputs $v(t)$ of the TCTM and interpret the external demand $w(t)$ as the new system input. Demand and supply functions are also held constant, for now. Using monotonicity of the auxiliary functions, we are ready to prove that the resulting system is monotone.

\begin{lemma} 
\label{lemma:tctm_monotonicity}
Assume $\phi_e^*(t)$ for $e \in \N$, $t \in 0, \dots T$ is a feasible input sequence in the deterministic CTM for the uncertainty realization $\omega = \overline \omega$, with corresponding state evolution $\rho^*(t)$. Consider now the TCTM with inputs chosen as $v^*(t) := \rho^*_e(t) - \Delta t \cdot \phi_e^*(t)$ and demand $d_{e,t}(\cdot) = \underline d_e(\cdot)$ and supply functions $s_{e,t}(\cdot) = \underline s_e(\cdot)$ fixed, and interpret the external demand $w(t) \in \Omega_{w,t}$ as the new system input. The resulting system is a monotone system, with monotone constraints.
\end{lemma} 
\begin{proof} 
The proof requires to verify monotonicity of each relevant function individually. To start, the system dynamics for $e \notin \N$ can be written as 
\begin{align*}
z_e(t+1) = \max \left\{ f_e^d \big( z(t), w(t) \big) ,~ \max_{i \in \E^+(e)} f_{i,e}^s \big( z(t), w(t) \big) \right\} + \Delta t \cdot p_e^\top \Big( w(t) + f^r \big( z(t), w(t) \big) \Big) .
\end{align*}
Positive sums of monotone functions are monotone, and the (pointwise) maximum of monotone functions is monotone. Since the auxiliary functions $f_e^d( \cdot )$, $f_{i,e}^s( \cdot )$ and $f^r( \cdot )$ are monotone in $z(t)$ and $w(t)$, and $\Delta t \geq 0$ and $p_e^\top \geq 0$, it follows that the system dynamics for cells $e \notin \N$ are monotone. Similarly, the dynamics for controlled cells are given as
\begin{align*}
z_e(t+1) = \max\big\{ z_e(t), v_e(t) \big\} + \Delta t \cdot p_e^\top \Big( w(t) + f^r \big( z(t), w(t) \big) \Big) .
\end{align*}
The arguments of the maximization are both monotone in $z(t)$ and $w(t)$, and the system dynamics for $e \in \N$ are monotone according to the same reasoning as before. It remains to verify monotonicity of the constraints. The demand constraints can be rewritten as $f_e^d \big( z(t), w(t) \big) - v_e(t) \leq 0$ and it is apparent that the LHS is monotone in $z(t)$ and $w(t)$. For analyzing the supply constraints, note that $\rho_e(t) = \frac{z_e(t)}{l_e}$ for all $e: \tau_e \in \M$. Thus, the constraints can be expressed as
\begin{align*}
\sum_{i \in \E} \beta_{e,i} \cdot \max \bigg\{ 0, \frac{z_i(t) - v_i(t)}{\Delta t} \bigg\} - s_e \left( \frac{z_e(t)}{l_e} \right) &\leq 0 & \forall e : \tau_e \in \M,
\end{align*}
and monotonicity of the LHS in $z(t)$ and $w(t)$ is apparent.
\end{proof} 

It should be emphasized that the system model used in Lemma \ref{lemma:tctm_monotonicity} is based on a traffic network with \emph{controlled} merging junctions. Control of merging flows is critical, as one can easily verify that an autonomous version of the TCTM which models merging junctions e.g.\ via the proportional-demand allocation rule is \emph{not} monotone in the states $z(t)$. Monotonicity simplifies accounting for uncertainty in the external demands $w(t)$, as we will see subsequently. It turns out that we can characterize the influence of the fundamental diagram in a similar manner.
\begin{observation} 
\label{observation:monotonicity}
For all $d_t( \cdot ) \in \Omega_{d}$ and $s_t( \cdot ) \in \Omega_{s}$, we have that
\begin{align*}
f \big( z(t), v(t), w(t), d_t( \cdot ), s_t( \cdot ) \big) &~\leq~ f \big( z(t), v(t), w(t), \underline d( \cdot ), \underline s( \cdot ) \big), \\
g \big( z(t), v(t), w(t), d_t( \cdot ), s_t( \cdot ) \big) &~\leq~ g \big( z(t), v(t), w(t), \underline d( \cdot ), \underline s( \cdot ) \big),
\end{align*}
that is, the lower bounds on supply and demand can be used to compute upper bounds on the system equations and the LHS of the constraints of the TCTM.
\end{observation}
The observation can easily be verified for each equation individually.

\subsection{Solution of the robust counterpart} \label{sec:main_result} 

The monotonicity properties of the TCTM can be leveraged in order to solve the robust FNC problem.
\begin{theorem} 
\label{theorem:robust_counterpart}
Consider the robust FNC problem over the horizon $\{0,1,\dots,T\}$, for a network with controlled merging flows satisfying Assumption \ref{assumption:graph} and with uncertainty sets satisfying Assumption \ref{assumption:uncertainty}. The robust FNC problem is equivalent to the convex optimization problem
\begin{equation} 
\label{eq:solution_counterpart}
\begin{array}{rrll}
\C^{**} := & \underset{\rho(t), \phi(t)}{\min} & \TTS \\ [1ex]
 & \text{subject to} & \text{Conservation law \eqref{eq:ctm_densities} and ramp constraints \eqref{eq:ramp_constraints} for $\omega = \overline \omega$.} \hspace{-3cm} & \hspace{3cm}  \\[1ex]
 & & 0 \leq \phi_e(t) \leq \underline{d}_e \big( \rho_e(t) \big) ~& \forall e \in \E  \\[1ex]
 & & \sum_{i \in \E} \beta_{e,i} \cdot \phi_i(t) \leq \underline{s}_e \big( \rho_e(t) \big) ~& \forall e \notin \mathcal{S} \\[1ex]
 & & \text{$\rho(0)$ given.}
 \end{array}
\end{equation}
with optimal value $\C^{**}$, in the sense that the optimal values are equal $\C^{**} = \C^*$, whenever the problem is feasible. An optimizer to the original problem is given by the policy $\phi_e(t) := \max \left\{ 0,~ \phi^*_e(t) + \frac{p_e^\top L \cdot (\rho(t) - \rho^*(t))}{\Delta t} \right\}$ for the minimizing player, where $\phi^*_e(t)$ and $\rho_e^*(t)$ are any optimizer of problem \eqref{eq:solution_counterpart}, and the choice $ \omega = \overline \omega$ for the maximizing player.
\end{theorem} 
Note that the nonlinear flow dynamics \eqref{eq:ctm_flows} have been relaxed and merged with the flow constraints \eqref{eq:ctm_constraints} for controlled flows. As a result, problem \eqref{eq:solution_counterpart} is exactly the convex relaxation of the deterministic FNC problem for the worst-case uncertainty realization $\omega = \overline \omega$.
\begin{proof} 
We will actually prove a slightly stronger statement. In particular, we will show that if the minimax problem is interpreted as a two-player zero sum game, the action $\omega = \overline \omega$ by the maximizing player and the suggested policy of the minimizing player form a pure \emph{Nash Equilibrium} (NE). Then, the theorem follows directly from evaluating the NE policies over the time horizon. In addition, this proof suggests that the order of play does not matter, i.e., the result of the minimax problem does not change if the maximizing player choses first and the minimizing player choses second. 

It is straightforward to show that the minimizing player cannot improve over $\C^{**}$, since
\begin{equation*} 
\begin{array}{rrll}
\C^{**} \geq& \underset{ \pi \in \Pi }{\min} & \TTS  \\ [1ex]
 & \text{subject to} & \text{CTM dynamics \eqref{eq:ctm_densities}, \eqref{eq:ctm_flows} and constraints \eqref{eq:ramp_constraints}, \eqref{eq:ctm_constraints} for $\omega = \overline \omega$.} \\[1ex]
 =& \underset{\rho(t), \phi(t)}{\min} & \TTS \\ [1ex]
 & \text{subject to} & \text{CTM dynamics \eqref{eq:ctm_densities}, \eqref{eq:ctm_flows} and constraints \eqref{eq:ramp_constraints}, \eqref{eq:ctm_constraints} for $\omega = \overline \omega$.} \\ [1ex]
 \geq& \C^*
\end{array}
\end{equation*}
With the optimization variables of the maximization problem fixed, there is no uncertainty present in the minimization problem any longer and hence, it is no longer necessary to optimize over policies but one can optimize over variables $\rho(t)$ and $\phi(t)$ instead. The last inequality follows since problem \eqref{eq:solution_counterpart} is a relaxation of the deterministic FNC problem with $\omega = \overline \omega$.

To show that the maximizing player cannot improve over $\mathcal{D}^*$ either, consider the trajectory $\rho^*(t)$ and $\phi^*(t)$ obtained as the solution to the (deterministic) optimization problem \eqref{eq:solution_counterpart}. According to \cite[Theorem 2]{schmitt2017exact}, there exists a solution $\tilde z(t)$ and $\tilde v(t)$ to the corresponding, non-relaxed, deterministic FNC problem \eqref{eq:fnc} (for $\omega = \overline \omega$) such that $\tilde \phi_e(t) = \phi^*_e(t)$ for all controlled cells $e \in \N$. According to Lemma \ref{lemma:equivalence_tctm}, this implies that there exists a solution $z^*(t)$, $v^*(t)$ for the optimization problem
\begin{align*}
\begin{array}{rrll}
& \underset{z(t), v(t)}{\min} & \sum_{t = 0}^{T} {\bf 1}^\top (\I - \underline{R} ) z(t) \\ [1ex]
 & \text{subject to} & \text{TCTM dynamics \eqref{eq:tctm_dynamics} and constraints \eqref{eq:tctm_constraints} for $\omega = \overline \omega$.} 
\end{array}
\end{align*}
such that $v^*_e(t) = p_e^\top L \cdot \tilde \rho(t) - \Delta t \cdot \tilde \phi_e(t)$. From earlier results, we know that  
\begin{align*}
v^*_e(t) &= p^\top_e L \rho(0) + \Delta t \cdot \sum_{\tau = 0}^{t-1} \bigg( p_e^\top(t) w(\tau) + \sum_{i \in \E} p_e(i) \cdot \Big( {\sum}_{j \in \N_S} \beta_{i,j} \tilde \phi_j(\tau) \Big) \bigg) - \Delta t \cdot \sum_{\tau = 0}^t \tilde \phi_e(\tau) \\
&= p^\top_e L \rho(0) + \Delta t \cdot \sum_{\tau = 0}^{t-1} \bigg( p_e^\top(t) w(\tau) + \sum_{i \in \E} p_e(i) \cdot \Big( {\sum}_{j \in \N_S} \beta_{i,j}\phi^*_j(\tau) \Big) \bigg) - \Delta t \cdot \sum_{\tau = 0}^t \phi^*_e(\tau) \\
&= p_e^\top L \rho^*(t) - \Delta t \cdot \phi_e^*(t)
\end{align*}
for all controlled cells $e \in \N$. For the minimizing player, we consider the candidate policy $\phi_e(t) := \max \left\{ 0,~ \phi^*_e(t) + \frac{p_e^\top L \cdot (\rho(t) - \rho^*(t)}{\Delta t} \right\} = \max \left\{ 0, \frac{z_e(t) - v^*_e(t)}{\Delta t} \right\}$, for all cells $e \in \N$, in the following denoted by $\pi^*_t$. For this candidate policy, we study the system evolution in terms of the CCTM and denote the resulting trajectory as $z(t)$. 
Clearly, $z(0) = z^*(0)$ and for $\omega = \overline \omega$, $z(t) = z^*(t)$ for all $t$. We now aim to show that $z(t) \leq z^*(t)$ by induction. To do so, assume $z(t) \leq z^*(t)$. It follows that
\begin{align*}
z(t+1) &= f \big( z(t), v^*(t), w(t), d_t( \cdot ), s_t( \cdot ) \big) \\[1ex]
 &\leq f \big( z^*(t), v^*(t), \overline w(t), d_t( \cdot ), s_t( \cdot ) \big) \\[1ex]
 &\leq f \big( z^*(t), v^*(t), \overline w(t), \underline d( \cdot ), \underline s( \cdot ) \big) = z^*(t+1) .
\end{align*}
Here, the first inequality follows from Lemma \ref{lemma:tctm_monotonicity} and the second inequality from Observation \ref{observation:monotonicity}. Consequently, feasibility of the trajectory $z^*(t)$ implies feasibility of the trajectory $z(t)$ since
\begin{align*}
g \big( z(t), v^*(t), w(t), d_t( \cdot ), s_t( \cdot ) \big) \leq g \big( z^*(t), v^*(t), \overline w(t), d_t( \cdot ), s_t( \cdot ) \big) \leq g \big( z^*(t), v^*(t), \overline w(t), \underline d( \cdot ), \underline s( \cdot ) \big) \leq 0.
\end{align*}
Hence, the candidate policy is always in the set of feasible policies, $\pi^*_t \in \Pi_t$. In addition, the objective 
\begin{equation*}
\TTS = \sum_{t = 0}^T \sum_{e \in \E} l_e \rho_e(t) = \sum_{t =0}^T \underbrace{ {\bf 1}^\top (\I - \underline{R} )}_{=: c^\top} z(t)
\end{equation*}
is monotone in $z(t)$, since $c_e = 1 - \sum_{i \notin \N} \beta_{e,i} \geq 1 - \sum_{i \in \E} \beta_{e,i}  \geq 0$. The latter inequality holds due to the traffic conservation law. Therefore,
\begin{equation*} 
\C^* = \min_{\pi \in \Pi} \left( \max_{\omega \in \Omega} ~ \TTS( \pi, \omega) \right) \leq \max_{\omega \in \Omega} ~ \TTS( \pi^*, \omega) \leq \TTS( \pi^*, \overline \omega ) = \C^{**}.
\end{equation*}
For the last equality, recall that the candidate policy replicates the solution to \eqref{eq:solution_counterpart} for $\omega = \overline \omega$. Combining the inequalities, we obtain $\C^*= \C^{**}$, which proves that the considered solutions form a NE, and therefore also a solution to the robust FNC problem.
 \end{proof} 

The proposed reformulation of the robust FNC problem greatly simplifies the optimization problem, since the optimization problem \eqref{eq:solution_counterpart} is no longer a nonconvex, bi-level optimization problem over infinite-dimensional policies. Instead, it is a finite-dimensional, convex optimization problem.

The policy $\pi^* \in \Pi$, from now on called the NE policy, is instrumental in the proof, since it leads to a pure NE, together with the worst-case uncertainty realization $\overline \omega$. However, it is not subgame-perfect. This means that even though it achieves minimal cost if the worst-case uncertainty realization occurs in every time step, it is suboptimal if other uncertainty realizations are encountered. The latter case is typical in practice, since the uncertain parameters are usually stochastic and modeling nature as a cost-maximizing antagonist is merely a tool used to design a robust control policy. To improve performance for non-worst-case uncertainty realizations, while still retaining robustness guarantees, one can use subgame-perfect strategies. By definition, a subgame-perfect strategy for the minimizing player in the robust FNC problem is obtained if problem \eqref{eq:solution_counterpart} is resolved at every time-step, for the observed state (which depends on the uncertainty realization in the previous time step) and with an appropriately truncated horizon. In control terms, this corresponds to introducing feedback instead of using essentially a feedforward approach.\footnote{One can argue that even the NE policy is a feedback strategy if expressed in terms of the densities, however, in TCTM variables it corresponds to pure feedforward control: $v_e(t) =  v^*_e(t)$, for all $e \in \N$, is chosen where $v^*_e(t)$ is computed by solving problem \eqref{eq:solution_counterpart} once.} A disadvantage of the subgame-perfect policy is the need to solve optimal control problems with long horizons at every time step. This might pose a computational challenge even if only convex optimization problems need to be solved, depending on the sampling time, the horizon length and the size of the road network. Therefore, we will also consider receding horizon policies in Section \ref{sec:receding_horizon}, which are designed to approximate the subgame-perfect policy with less computational effort, while retaining the robustness guarantees.

Before proceeding, we consider the implications of the previous result for deterministic settings, where we can establish that the optimal TTS of the FNC with controlled merging junction is monotone in external demand and demand and supply functions, in the following sense:
\begin{corollary} \label{corollary:monotonicity_objective}
Consider the deterministic FNC problem for a network with controlled merging junctions for $0\leq t \leq T$ and denote its optimal value as a function of the (known) uncertainty realization $\omega$ as $\TTS( \omega )$. Assume the network satisfies Assumption \ref{assumption:graph} and for the uncertainty realization $\overline \omega$, Assumption \ref{assumption:fd} is satisfied. Then, for all $\omega \in \Omega$, where $\Omega$ is bounded by the worst-case uncertainty realization $\overline \omega$, it holds that $\TTS( \omega ) \leq \TTS( \overline \omega )$.
\end{corollary}
This result follows immediately from from Theorem \ref{theorem:robust_counterpart}. Note that whereas Assumption \ref{assumption:fd} is required to hold for $\overline \omega$, that is, for the lower bounds on demand $\underline d_{e}(\cdot)$ and supply function $\underline s_{e}(\cdot)$, it is \emph{not} required for $\omega$ ($d_{e,t}(\cdot)$ and $s_{e,t}(\cdot)$, respectively). This means that we can use Corollary \ref{corollary:monotonicity_objective}, together with suitably defined lower bounds on demand and supply functions, in order to compute upper bounds on the optimal TTS in traffic networks for which Assumption \ref{assumption:fd} is not satisfied, for example, if demand functions with a capacity drop are considered.

\section{Discussion and extensions} \label{sec:discussion} 

In this section, we first explore the limitations of the results and demonstrate via counterexamples that Theorem \ref{theorem:robust_counterpart} does not extend to entirely uncontrolled merging junctions described e.g.\ by the priority-proportional merging rule. Subsequently, we introduce two extensions of our main result. First, we show in Section \ref{sec:ramp_metering} that the results extend to freeways controlled via ramp metering, where only the merging flow from the onramp is controlled, if an asymmetric merging model for onramp merging junctions is used. Second, we explore in Section \ref{sec:receding_horizon} how receding horizon control can be used to approximate the subgame-perfect solution, without the need to solve a ``large" optimization problem at every time step.

\subsection{Counterexamples} 
\label{sec:counterexamples}
Theorem \ref{theorem:robust_counterpart} relies on explicitly identifying the worst-case FD and external demand and only holds as long as merging junctions are controlled. If merging flows are uncontrolled and, for example, modeled by the demand-proportional rule, one can construct counterexamples that show that the minimal FD and maximal external demand are \emph{not} the worst case uncertainty realization. The first counterexample demonstrates that for networks with uncontrolled merging junctions, an increase in external demand can in fact lead to an overall decrease in TTS.

\begin{example} \label{example:ex1} 
\begin{figure}[t] 
	\centering
	\begin{subfigure}[b]{0.33\textwidth}
		\includegraphics[width=\textwidth]{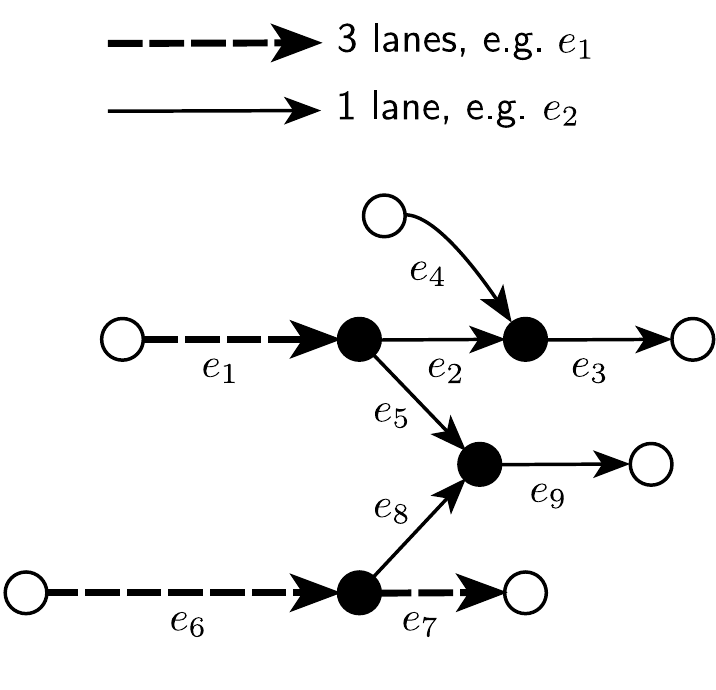}
		\caption{Network structure.}
		\label{fig:counterexample1_graph}
	\end{subfigure} \hspace{0.5cm}	
	\begin{subfigure}[b]{0.6\textwidth}
		\includegraphics[width=\textwidth]{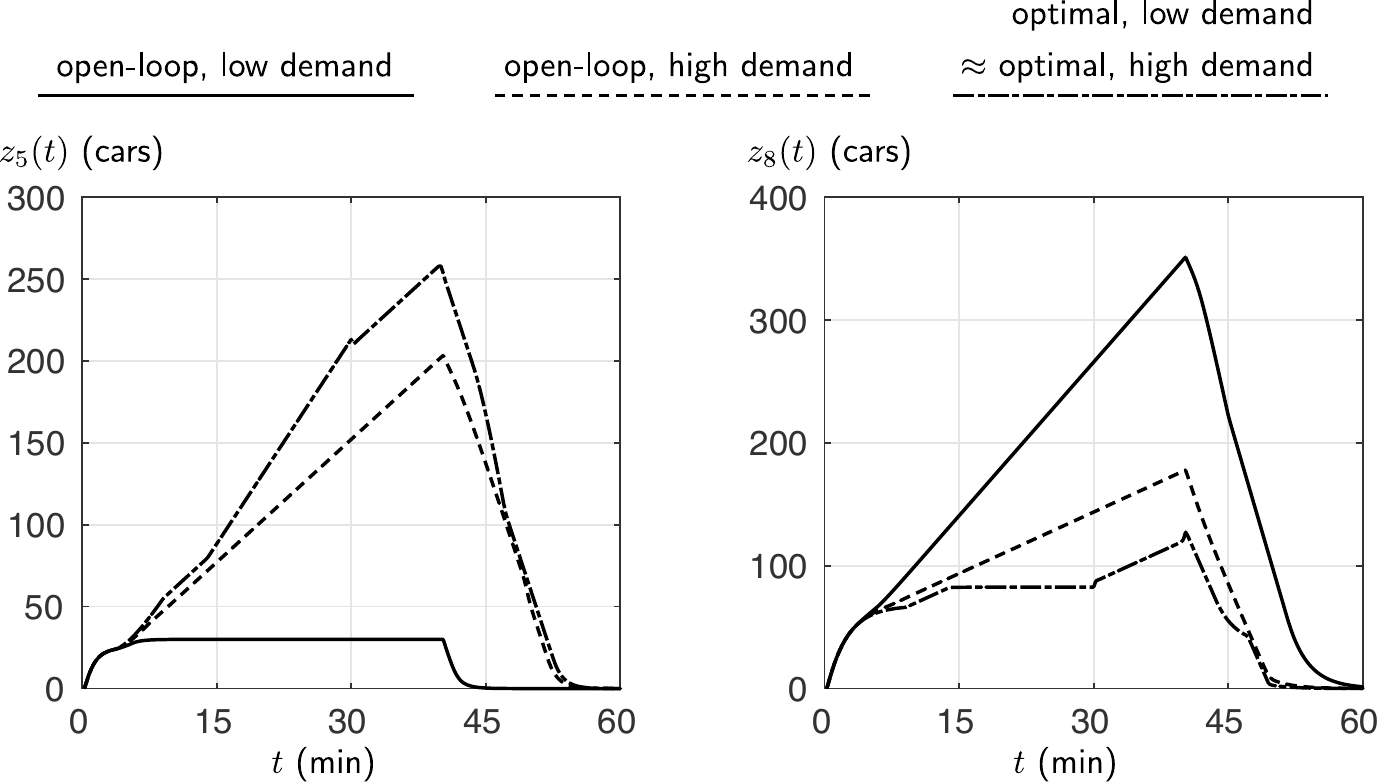}
		\caption{Evolution of the backlog $z_e(t)$ in cells $e_5$ and $e_8$.}
		\label{fig:counterexample1_plots}
	\end{subfigure}
	\caption{In a network with both FIFO-diverging junctions and uncontrolled merging junctions (Figure \ref{fig:counterexample1_graph}), an increase in external demand may lead to a decrease in TTS. In this particular example, increasing the external demand $w_{4}(t)$ in the uncontrolled network brings the evolution of the backlog of cells $e_5$ and $e_8$ closer to the optimal trajectories in the controlled case, as depicted in Figure \ref{fig:counterexample1_plots}.}
	\label{fig:counterexample1}
\end{figure} 
Consider the network depicted in Figure \ref{fig:counterexample1_graph}. The length of the cells is $l_1 = 2$km, $l_6 = 4$km and $l_e = 0.5$km otherwise\footnote{For ease of notation, we will index quantities as $l_1$, $l_2$ \dots instead of $l_{e_1}$, $l_{e_2}$ \dots from now on, whenever cells are denoted as $e_{1}, e_{2}, \dots$.}. Cells consist of either three or one lane, as indicated in the figure. The demand function of each lane is given as $d(\rho) = \min \{ v \rho(t), F\}$, with free-flow speed $v = 100$km/h and lane capacity $F = 2000$cars/h. The supply function of each lane is given as $s(\rho) = \min \{ F, (\bar \rho - \rho(t)) w \}$, with traffic jam density $\bar\rho = 120$cars/km per lane and congestion wave speed $w = 20$km/hs. The diverging junctions are modeled using the FIFO model with $\beta_{2,1} = 2/3$, $\beta_{5,1} = 1/3$, $\beta_{7,6} = 3/4$ and $\beta_{8,6} = 1/4$. We aim to compare the system evolution for the case when the merging junctions upstream of cells $e_9$ and $e_3$ are controlled optimally with the system evolution when the network is uncontrolled. In the latter case, we adopt the demand-proportional merge model used e.g.\ in \cite{coogan2016stability}, in which the supply of free space of a merge is allocated proportional to the demands of the upstream cells, whenever total demand exceeds supply. In particular,
\begin{equation*}
\phi_5(t) = d_5 \big( \rho_5(t) \big) \cdot \min \left\{ 1,~  \frac{ s_9 \big( \rho_9(t) \big) }{ d_5 \big( \rho_5(t) \big) + d_8 \big( \rho_8(t) \big) } \right\}.
\end{equation*}
Merging flows $\phi_8(t)$, $\phi_2(t)$ and $\phi_4(t)$ are computed analogously. We aim to compare the system evolution in a ``low-demand" scenario, where $w_4(t) \equiv 0$ with a ``high-demand" scenario, where $w_4(t) = 600$cars/h for $0 \leq t \leq 40$min. In both cases, the external demands $w_1(t) = 3000$cars/h and $w_6(t) = 6000$cars/h are constant for $0 \leq t \leq 40$min and zero afterwards. The system evolution is studied for a horizon of $T = 60$min, after which the network is almost completely empty in both scenarios. In accordance with Theorem \ref{theorem:robust_counterpart} and Corollary \ref{corollary:monotonicity_objective}, we find that the TTS for the low-demand scenario $\TTS^*( \underline w(t) ) = 1295$h is smaller than the one for the high-demand scenario $\TTS^*( \overline w(t) ) = 1350$h if the network is controlled optimally. However, if the network is uncontrolled (in ``open-loop"), it turns out that an \emph{increase} in the external demand $w_4(t)$ may lead to a \emph{decreases} in TTS, from $\TTS^{\text{ol}}( \underline w(t) ) = 2373$h in the low-demand scenario to $\TTS^{\text{ol}}( \overline w(t) ) = 1652$h in the high-demand scenario.

This effect can be understood with the help of Figure \ref{fig:counterexample1_plots}, which depicts the evolution of the backlog $z_5(t) = l_5 \rho_5(t) + \beta_{5,1} l_1 \rho_1(t)$ and $z_8(t) = l_8 \rho_8(t) + \beta_{8,6} l_6 \rho_6(t)$. Flows from cells $e_5$ and $e_8$ compete for the supply of free space of cell $e_9$, which is the major bottleneck in the low demand case. Figure \ref{fig:counterexample1_plots} shows that the optimal control policy gives priority to the flow from cell $e_8$ (leading to small $z_8(t)$), to avoid propagation of congestion into cell $e_6$, which serves a large proportion of the total traffic demand. By contrast, equal priority is given to both flows into cell $e_9$ in the uncontrolled network. For the low-demand scenario, this means that the (comparatively smaller) demand from cell $e_5$ can be served completely, while a congestion forms in cells $e_8$ and $e_6$ (leading to large backlog $z_8(t)$). In particular, traffic seeking to travel from cell $e_6$ to $e_7$ is also obstructed. By contrast, the additional, external demand $w_4(t)$ in the high-demand scenario causes congestion in cell $e_2$, which blocks flow into cell $e_5$ via the FIFO-diverging junction. Hence, a large percentage of the traffic demand from cell $e_8$ flows unobstructed, much in the same way as if the merging junction upstream of cell $e_9$ was controlled optimally. Intuitively speaking, the demand $w_4(t)$ acts like a switch which partially alleviates the need for controlling the merging junction and hence, the performance in the high-demand scenario is improved in comparison to the low-demand scenario when the network is uncontrolled.
\end{example}

A similar example using Daganzo's priority rule can be constructed, if the merging priorities are chosen appropriately. Due to this counterexample, it is clear that Theorem \ref{theorem:robust_counterpart} and Corollary \ref{corollary:monotonicity_objective} do not extend to networks with uncontrolled merging junctions.

The second counterexample demonstrates in a similar manner that a reduction in free-flow velocity, i.e.\, a smaller demand function, can be beneficial if demand functions are not monotone. 
\begin{example} \label{example:ex2} 
\begin{figure}[t] 
	\centering
	\begin{subfigure}[b]{0.52\textwidth}
		\includegraphics[width=\textwidth]{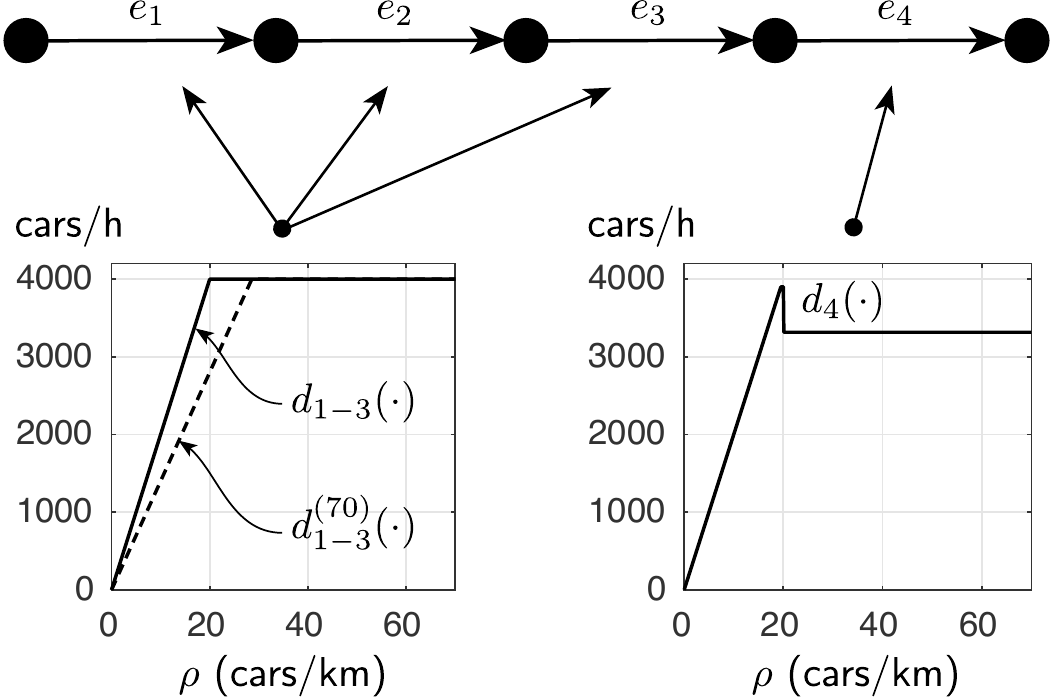}
		\caption{The network topology is a line, modeling a freeway. The demand functions of cell $e_4$ models a capacity drop.}
		\label{fig:counterexample2_graph}
	\end{subfigure} \hspace{1cm}	
	\begin{subfigure}[b]{0.35\textwidth}
		\includegraphics[width=\textwidth]{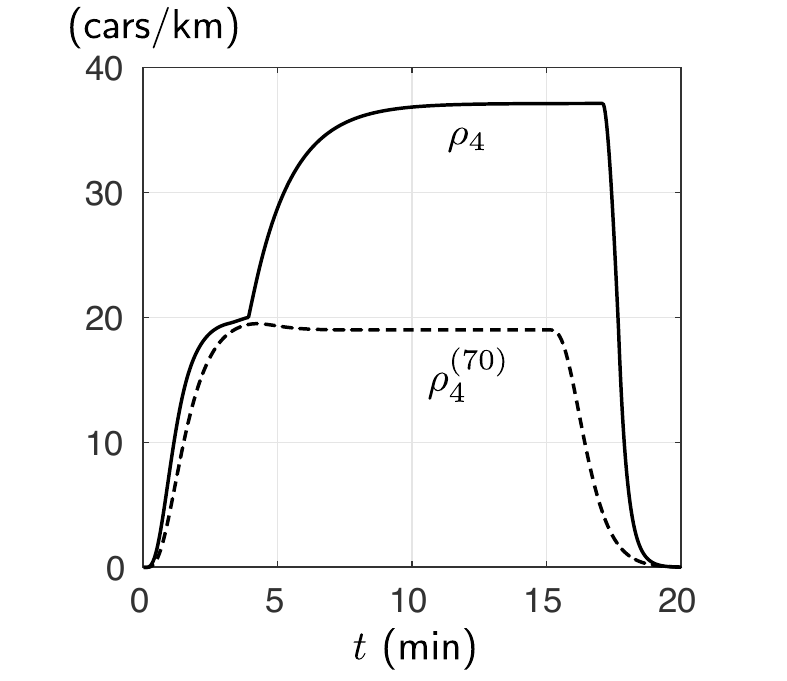}
		\caption{Evolution of $\rho_4(t)$. In the case without speed limit, the critical density ($20$cars/km) is exceeded and the capacity drop comes into effect.}
		\label{fig:counterexample2_plots}
	\end{subfigure}
	\caption{In a network with non-monotone demand functions, in particular a capacity drop, a reduction in free-flow velocity can improve TTS.}
	\label{fig:counterexample2}
\end{figure} 
Consider the network depicted in Figure \ref{fig:counterexample2_graph}. All cells are $l_e = 0.5$km long and consist of two lanes each. The demand function of each lane in cells $e_{1-3}$ is the same, $d(\rho(t)) = \min \{ v \rho(t), F \}$, with free-flow speed $v = 100$km/h and lane capacity $F = 2000$cars/h. The supply function of each lane in all cells $e_{1-4}$ is given as $s(\rho(t)) = \min \{ F, (\bar\rho - \rho(t)) w \}$, with traffic jam density $\bar\rho = 120$cars/km per lane and congestion wave speed $w = 25$km/h. We also consider the case, when a speed limit of $70$km/h is introduced for cells $e_{1-3}$, leading to smaller demand functions $d_{1-3}^{(70)}(\cdot)$ in comparison the baseline scenario, as depicted in Figure \ref{fig:counterexample2_graph}. Parameters other than the free-flow velocity remain unchanged. Cell $e_4$ has a lower capacity than the first three cells. Its maximal demand is $d_4( 20 \text{cars/km} ) = 3900$  and in addition, its demand function experiences a capacity drop with $15\%$ capacity reduction, as soon as a critical density ($20$cars/km) is exceeded. We will compare the system performance in terms of TTS with and without the speed limit. Note that even though demand and supply functions are ordered, in the sense that $d_{1-3}(\rho) \geq d_{1-3}^{(70)}(\rho)$ for all $\rho \geq 0$, Corollary \ref{corollary:monotonicity_objective} is \emph{not} applicable, since the minimal demand function $\underline d_4(\cdot) = d_4(\cdot)$ does not satisfy Assumption \ref{assumption:fd}. 

The system does not contain any junctions and therefore, it is uncontrolled. The external demand equals $w_1(t) = 3800$cars/h for $0\text{min } \leq t < 2$min and $4\text{min } \leq t < 15$min. During a short time $2\text{min } \leq t < 4$min, the external demand increases to $w_1(t) = 3950$cars/h. For $t \geq 15$min, it equals zero. The evolution of the density in cell $e_4$ is depicted in Figure \ref{fig:counterexample2_plots}. Note that for $2\text{min } \leq t < 4$min, the external demand exceeds the capacity of cell $e_4$. In the case with $v_{1-3} = 100$km/h, this excess demand leads to the density $\rho_4(t)$ exceeding the critical density of $20$cars/km at $t \approx 4$min. Consequently, the capacity drop comes into effect and a congestion forms. Note that while the capacity drop is in effect, not even $3800$cars/h can be served, which means that the congestion queue expands until $t = 15$min, extending into cells $e_3$ and $e_2$. In this case, $\TTS = 15.3$h. By contrast, reducing the free-flow velocity to $v_{1-3} = 70$cars/h slows down the propagation of the increased demand and ``smooths" the traffic flow reaching cell $e_4$. As a consequence, the critical density in cell $e_4$ is never exceed and no congestion forms, leading to an improvement in TTS, with $\TTS^{(70)} = 12.6$h. 
\end{example}
It should be emphasized that in both examples, the parameters have been chosen to amplify the effect on TTS. Determining how relevant these effects are in the real world requires further, empirical work. Examples for which Theorem \ref{theorem:robust_counterpart} and Corollary \ref{corollary:monotonicity_objective} do hold are presented in the numerical study in Section \ref{sec:numerical}.

\subsection{Extension to ramp metering} \label{sec:ramp_metering} 

An important application of the FNC problem is freeway ramp metering, where traffic lights are installed on onramps to control the flow of vehicles onto the freeway mainline. The objective is to prevent, delay or reduce congestion on the mainline, to improve bottleneck flow and prevent obstruction of upstream off-ramps by the congestion queue. In ramp metering, the flow from the onramp is controlled but the flow on the mainline is not. Hence, the controlled junction model introduced in Section \ref{sec:problem} is not suitable to model such a situation. Previous work has shown that if an \emph{asymmetric} merging model \cite{gomes2006optimal,gomes2008behavior} is used for such onramps, the convex relaxation of the FNC problem remains tight \cite{schmitt2017exact}. The disadvantage of such a model is that it requires the additional assumption that mainline congestion does not obstruct the metered onramp flow, an assumption that is not satisfied for all metered onramps in reality. However, we can show that if such an assumption is satisfied and the asymmetric merging model is used, all previous results generalize.

Consider an \emph{asymmetric junction}, where exactly two upstream cells $e$ and $i$ merge into one downstream cell $j$. For simplicity, we assume $\beta_{j,e} = \beta_{j,i} = 1$. One upstream cell $e$ models the onramp and its flow $\phi_e(t)$ is controlled, subject to the constraints $0 \leq \phi_e(t) \leq d_{t,e} \big( \rho_e(t) \big)$. It is \emph{not} explicitly constrained by the supply of free space in cell $j$, but instead, we make the a priori assumption that $d_{t,e} \big( \rho_e(t) \big) \leq s_{t,j} \big( \rho_j(t) \big)$, that is, the supply constraint never becomes active (see Proposition \ref{proposition:extension_onramps}). Such an assumption is hard to verify a priori, but one can resort to a posteriori verification to check validity of a particular solution. Details are provided in \cite{schmitt2017exact}. This assumption alleviates the need for control of the second upstream cell $i$, called the mainline cell, whose flow $\phi_i(t)$ is computed as $\phi_i(t) = \min \left\{ d_{t,i} \big( \rho_i(t) \big), s_{t,j} \big( \rho_j(t) \big) - \phi_e(t) \right\}$. Note that $\phi_i(t) \geq 0$ because of the a priori assumption. 

\begin{proposition} \label{proposition:extension_onramps} 
If in addition to controlled junctions, asymmetric junctions are present in a network satisfying Assumptions \ref{assumption:graph}, then Theorem \ref{theorem:robust_counterpart} remains valid for the FNC problem with initial state $\rho(0)$ and uncertainty set $\Omega$ (satisfying Assumption \ref{assumption:uncertainty}) if $d_{t,e} \big( \rho_e(t) \big) \leq s_{t,j} \big( \rho_j(t) \big)$ holds for all onramp cells $e$ and cells $j$ downstream of the asymmetric junction, for all $\omega \in \Omega$, for any feasible control policy and for all $t \geq 0$.
\end{proposition} 
\begin{proof} 
For a traffic network which contains onramp junctions, define the reduced turning matrix $\rR$ as $\rR_{i,j} = \beta_{i,j}$ if $\tau_j$ is not a merging junction, or if $\tau_j$ is an asymmetric junction and cell $i$ is the mainline. Using the reduced routing matrix, the state transformation \eqref{eq:transformation} is performed as before. 

To show that Theorem \ref{theorem:robust_counterpart} still holds, it is sufficient to verify that Lemma \ref{lemma:tctm_monotonicity} and Observation \ref{observation:monotonicity} hold for the TCTM with asymmetric junctions, as the proof of Theorem \ref{theorem:robust_counterpart} only relies on the monotonicity properties that are verified in the intermediate results and on \cite[Theorem 2]{schmitt2017exact}, which applies to networks with asymmetric junctions. To show that Lemma \ref{lemma:tctm_monotonicity} holds for the TCTM with asymmetric junctions, one can extend the original proof, where monotonicity of each relevant function was verified individually. In particular, the dynamics and the demand constraints of controlled onramp flows into an asymmetric junction are identical to those of merging flows in controlled junctions and hence, they are monotone. In addition, one can verify that the dynamics of cells not immediately upstream of an asymmetric junction do not change. Therefore, it only remains to analyze monotonicity of the dynamics of mainline cells $i$ immediately upstream of an asymmetric junction. We find that
\begin{align}
\label{eq:asymmetric_dynamics}
z_i(t+1) = \max \left\{ f_i^d \big( z(t) \big) ,~ f_{j,i}^s \big( z(t) \big) + \Delta t \cdot \phi_e(t) \right\} + \Delta t \cdot p_i^\top \Big( w(t) + f^r \big( z(t), v(t) \big) \Big) ,
\end{align}
where $e$ is the controlled onramp and $j$ is the cell immediately downstream of the asymmetric junction. Monotonicity of the auxiliary functions $f_i^d(\cdot)$, $f_{j,i}^s(\cdot)$ and $f^r(\cdot)$ in $z(t)$ has been shown in Lemma \ref{lemma:auxiliary}. The flow from the onramp $\phi_e(t) = \max \left\{ 0, \frac{z_e(t) - v^*_e(t)}{\Delta t} \right\}$ is controlled and monotone in the state $z_e(t)$ and hence, the dynamics are monotone in $z(t)$ and $w(t)$ as well. Verifying Observation \ref{observation:monotonicity} for \eqref{eq:asymmetric_dynamics} is straightforward.
\end{proof} 
In the following sections, in particular in the numerical evaluation in Section \ref{sec:numerical}, we will allow for asymmetric junctions in ``networks with controlled junctions" and if asymmetric junctions are present, we assume that the supply condition specified in Proposition \ref{proposition:extension_onramps} is satisfied, whenever a theoretical statement is made. For numerical examples, we resort to an a posteriori verification of this condition for the computed solution instead. Also, we assume that the set $\N$ contains all controlled flows, that is, all merging flows into controlled junctions (as before) and in addition, the onramp flows into asymmetric junctions, but not the mainline flows into asymmetric junctions.

\subsection{Receding horizon control} \label{sec:receding_horizon} 

A potential problem of a policy based on solving the robust counterpart lies in its conservativeness. If the uncertainty sets are large, solutions computed for the worst-case disturbance might lead to performance deterioration if the actual uncertainty realizations are less extreme. One can introduce feedback to reduce the adverse effects. One such example is the subgame-perfect solution described before, where problem \eqref{eq:solution_counterpart} is re-solved at every time step $t$, for the entire remaining horizon $\{t,t+1,\dots,T\}$. Solving the entire problem at every time step is computationally expensive, though. Instead, one can attempt to approximate the optimal solution by only solving the optimal control problem for a shorter control horizon $\{t,\dots, t+T_c\}$, with $T_c \ll T$, at every time $t$. This approach is known as \emph{model predictive control} (MPC) or receding horizon control \cite{borrelli2017predictive}. However, performance degradation may result from choosing $T^c$ too short if no additional precautions are taken. Typically, a terminal cost or a terminal constraint is introduced to compensate for a short control horizon. In the following, we design a receding horizon control policy for the robust FNC problem. 
\begin{proposition} 
\label{proposition:mpc}
Consider the robust FNC problem over the horizon $\{0,1,\dots,T\}$, for a network with controlled merging flows satisfying Assumption \ref{assumption:graph} and with uncertainty sets satisfying Assumption \ref{assumption:uncertainty}. Let $\rho^*(t)$ and $\phi^*(t)$ denote an optimizer of the deterministic FNC \eqref{eq:fnc} for $\bf \omega = \overline \omega$, solved once for the complete horizon $t \in \{0,1,\dots,T\}$. Consider the receding horizon policy based on the convex optimization problem
\begin{equation} 
\label{eq:mpc}
\begin{array}{rrll}
 & \underset{\phi( \tau ), \rho( \tau )}{\text{minimize}} & \sum_{ \tau = t+1 }^{\max \{ t+T_c, T \}} l^\top \rho(\tau)  \\ [2ex]
 & \text{subject to} & \text{FNC conservation law \eqref{eq:ctm_densities} and ramp constraints \eqref{eq:ramp_constraints} for $\omega = \overline \omega$.} \hspace{-2cm} \\[1ex]
 & & \phi_e( \tau ) \leq \underline d_e \big( \rho_e( \tau ) \big) &\forall e \in \E \\[1ex]
 & &  \sum_{i \in \E} \beta_{i,e} \cdot \phi_i(t) \leq \underline{s}_e \big( \rho_e(t) \big) ~& \forall e \notin \mathcal{S} \\[1.5ex]
 & & P L \cdot \big( \rho(t+T_c) - \rho^*(t+T_c) \big) \leq 0 & \forall e \in \E \\[1ex]
 & & \text{Initial state $\rho( t )$ given.}
\end{array}
\end{equation}
The control policy selects the controlled flows $\phi_e(t)$ for $e \in \N$ equal to the values of the minimizer $\phi^*_e(t)$ of the optimization problem. This policy is well-defined\footnote{Here, ``well-defined" means that all optimization problems are feasible, as opposed to feasibility of the trajectory when the computed control inputs are implemented. This distinction is relevant since problem \eqref{eq:mpc} uses a relaxation of the system dynamics. The policy, that is, the optimizer of \eqref{eq:mpc}, is not necessarily unique, but the result holds regardless of which optimizer is used.} , feasible and achieves total cost (TTS) $\C^*_{\text{MPC}} \leq \C^*$. 
\end{proposition}
The proof of this proposition is postponed after some preliminary discussion and an additional, auxiliary result. The proposed receding horizon policy requires the one-time solution of a deterministic FNC problem over the complete horizon $\{0,1,\dots,T\}$ in a preliminary step to obtain $\rho^*(t)$. However, only optimization problems with short horizon length $T^c$ need to be solved at every time step at runtime. Just as in Theorem \ref{theorem:robust_counterpart}, the worst-case uncertainty realization $\overline \omega$ and relaxed flow constraints are used in \eqref{eq:mpc}. However, constraints on the terminal state $\rho(t+T_c)$ have been added. Because of these additional constraints, an additional, auxiliary result is needed, to ensure that relaxing the constraints does not change the optimal value of the optimization. Since we will use monotonicity of the TCTM in the proof of Proposition \ref{proposition:mpc} subsequently, we will also state this auxiliary lemma in terms of the TCTM.
\begin{lemma} \label{lemma:relaxation} 
The receding horizon policy defined by \eqref{eq:mpc} is identical to the policy which solves the non-convex problem
\begin{equation} 
\label{eq:mpc_nonconvex}
\begin{array}{rrll}
& \underset{ z(\tau), v(\tau) }{\text{minimize}} & \sum_{\tau=t+1}^{t+T} {\bf 1}^\top (\I - \rR) z( \tau ) \\ [2ex]
 & \text{subject to} & \text{TCTM dynamics \eqref{eq:tctm_dynamics} and constraints \eqref{eq:tctm_constraints} for $\omega = \overline \omega$.} \hspace{-2cm} \\[1ex]
 & & z_e( t+T ) - z^*_e( t+T ) \leq 0 & \forall e \in \E \\[1ex]
 & & \text{Initial state $z( t )$ given.}
\end{array}
\end{equation}
at every time $t$ and applies the controlled flows $\phi_e(t) = \frac{z^*_e(t) - v^*_e(t)}{\Delta t} $. Here, $v^*(t)$ and $z^*(t)$ are part of an (any) optimizer of \eqref{eq:mpc_nonconvex}. 
\end{lemma} 
The proof uses a variant of \cite[Theorem 2]{schmitt2017exact} and is provided in \ref{appendix:relaxation2}. It should be emphasized that the result does \emph{not} allow for arbitrary terminal constraints, in particular equality constraints. We are now ready to prove Proposition \ref{proposition:mpc} with arguments reminiscent of standard proofs for performance and recursive feasibility in MPC.

\begin{proof} 
According to Lemma \eqref{lemma:relaxation}, the policies defined by problems \eqref{eq:mpc_nonconvex} and \eqref{eq:mpc} are identical. Therefore, we can analyze the performance of the policy based on solving \eqref{eq:mpc_nonconvex} and make use of monotonicity of the TCTM. First, we define the \emph{predicted trajectory} $z_t( \tau )$ at time $t$. For $0 \leq \tau \leq t$, the predicted trajectory lies in the past and it equals the actual, observed trajectory. For the control horizon $t < \tau \leq t+T^c$, the predicted trajectory is defined to be equal to the solution obtained by solving problem \eqref{eq:mpc_nonconvex} at iteration $t$. Finally, $z_t(\tau) := z^*(\tau)$ for $t+T^c < \tau \leq T$.\footnote{Note that as opposed to the reference trajectory $z^*(\tau)$, the predicted trajectory is not necessarily feasible in the TCTM dynamics, due to the way in which $z_t(t+T^c)$ and $z_t(t+T^c+1)$ are defined.} With the help of the predicted trajectory, we can define the predicted cost as
\begin{equation*}
\C(t) := \sum_{\tau = 0}^{ T } {\bf 1}^\top (\I - \rR) z_t(\tau) , 
\end{equation*}
with ${\bf 1}^\top (\I - \rR) =: c^\top \geq 0$, as stated before. We employ induction to show that the predicted cost never increases, that is, $\C^* \geq \C(0)$ and $\C(t) \geq \C(t+1)$. At $t = 0$, the candidate solution $\hat z( \tau ) := z^*( \tau )$ for $0 \leq \tau \leq T$ is feasible in the optimization problem \eqref{eq:mpc_nonconvex}, since the reference trajectory $z^*(\tau)$ is feasible by Assumption. This implies that the solution $z_0( \tau )$ to \eqref{eq:mpc_nonconvex} realizes a cost over the horizon $0 \leq \tau \leq T_c$ smaller than the cost incurred by the reference trajectory. 
For $t < \tau \leq T$, we have that $z_0(\tau) = z^*(\tau)$. The start of the induction
\begin{align*}
\C(0) &= \sum_{\tau = 0}^{T} c^\top z_0(\tau) \leq \sum_{\tau = 0}^{ T} c^\top z^*(\tau) = \C^*
\end{align*}
follows. 

\newcommand*\mystrut[1]{\vrule width0pt height0pt depth#1\relax}

Consider now time $t > 0$ and the candidate solution $\hat z(\tau ) = z_{t-1}(\tau )$ for $t \leq \tau \leq t+T-1$. For this interval, feasibility of the candidate solution in the optimization \eqref{eq:mpc_nonconvex} at time $t$ follows from feasibility of the solution to the optimization \eqref{eq:mpc_nonconvex} in the previous time step. The terminal state $\hat z(t+T)$ of the candidate trajectory is obtained by simulating the TCTM for one step, using the input of the reference trajectory $v(t+T-1) = v^*(t+T-1)$, which implies feasibility with respect to the TCTM dynamics \eqref{eq:tctm_dynamics}. Furthermore, $\hat z(t+T) \leq z^*(t+T)$, since $\hat z(t+T-1) \leq z^*(t+T-1)$ due to the terminal constraint in the previous time step and monotonicity of the TCTM. This implies that the candidate trajectory satisfies the terminal constraint of \eqref{eq:mpc_nonconvex}. Also, feasibility with respect to the TCTM constraints \eqref{eq:tctm_constraints} for the terminal state $\tau = t+T$ follows by monotonicity of these constraints since
\begin{align*}
g \big( z_t( t+T ) , v^*( t+T ), \omega_{t+T} \big) \leq g \big( z^*( t+T ) , v^*( t+T ), \overline \omega_{t+T} \big) \leq 0 \quad \forall ~\omega_{t+T} \in \Omega_{T+t}.
\end{align*}
We conclude that the candidate trajectory is feasible in the optimization problem at iteration $t$, which implies that the cost achieved by solving problem \eqref{eq:mpc_nonconvex} at time $t$ improves over the cost in the candidate trajectory for the control horizon, that is, $\sum_{\tau = t}^{t+T-1} c^\top z_t(\tau ) \leq \sum_{\tau = t}^{t+T-1} c^\top z_{t-1}(\tau )$. We can now verify that
\begin{align*}
\C(t) &~=~ \sum_{\tau = 0}^{t-1} c^\top \underbrace{\mystrut{3ex}  z_{t}(\tau) }_{ \hspace{-1cm}  = z_{t-1}(\tau) \hspace{-1cm}  }
 ~+ \underbrace{ \sum_{\tau = t}^{t+T-1} c^\top z_{t}(\tau ) }_{ \leq \sum_{\tau = t}^{t+T-1} c^\top z_{t-1}(\tau) }
 ~+ c^\top \underbrace{ \mystrut{3ex} z_t(t+T) }_{ \hspace{-1cm} \mystrut{3ex} \leq z^*(t+T) = z_{t-1}(t+T)  \hspace{-1cm} }
 ~+ \sum_{\tau = t+T+1}^{\mathcal{T}} c^\top \underbrace{ \mystrut{3ex} z^*(\tau) }_{\hspace{-1cm}  = z_{t-1}(\tau) \hspace{-1cm}  } 
 ~\leq~ \C(t-1) ,
\end{align*}
where we use monotonicity of the TCTM cost in addition to previous, intermediate results. For $\tau > T - T^c$, the horizon of the optimization \eqref{eq:mpc_nonconvex} is truncated, but the same arguments can be used to certify a decrease of the predicted cost.
\end{proof}

We will evaluate the performance of such a receding horizon controller in the following chapter. In particular, we will compare its performance to the subgame-perfect solution and to a ``naive" receding horizon control approach without any terminal cost.

\section{Numerical study} \label{sec:numerical} 

\begin{figure} 
	\centering
	\includegraphics[width = \textwidth]{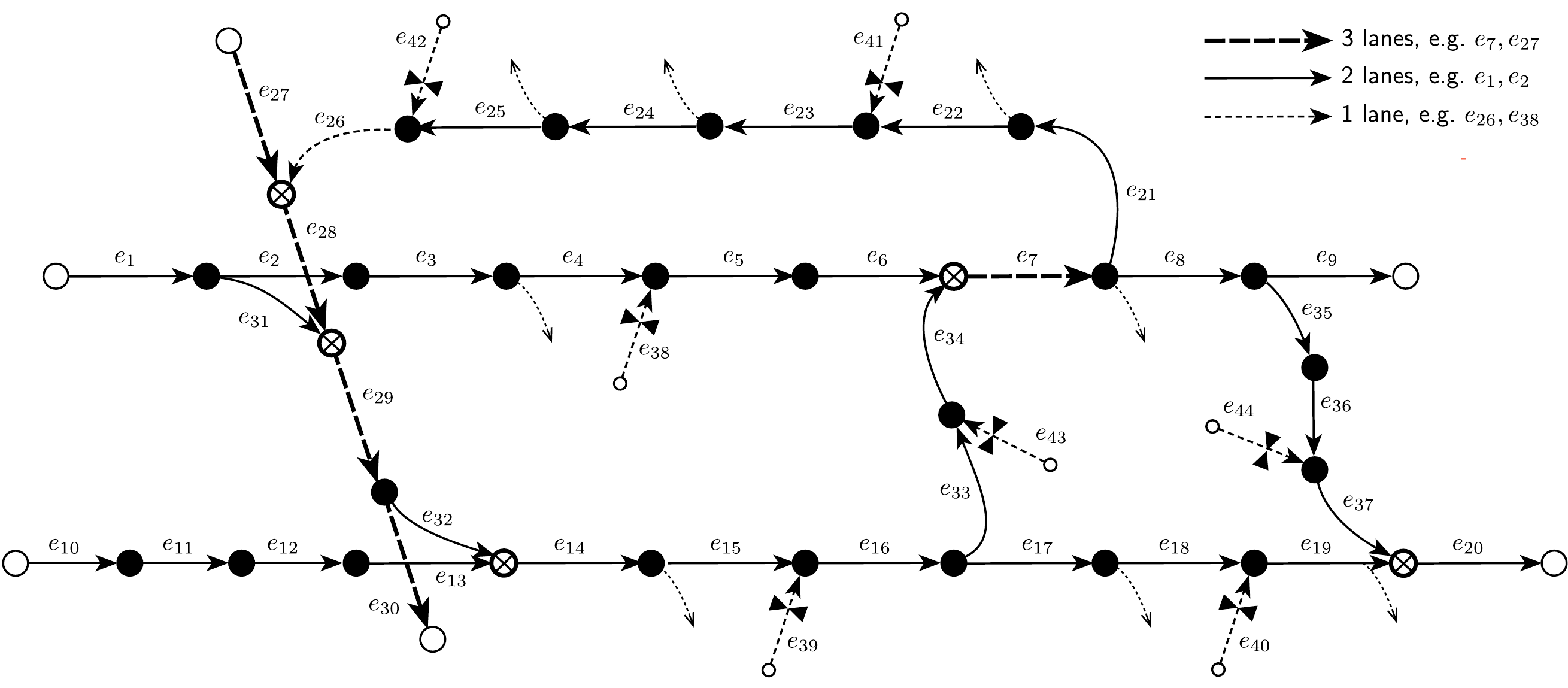}
	\caption{Network topology used in the numerical study. The capacity of each individual cell is proportional to the number of lanes. Depicted is the worst-case capacity of cells. We will also study cases when the capacity of cells $e_{10-16}$, $e_{21-26}$ and $e_{33-34}$ is increased in comparison to this depiction.}
	\label{fig:topology}
\end{figure}

In this section, we aim to verify the theoretical results via numerical simulations. To this end, we consider the artificial freeway network depicted in Figure \ref{fig:topology}. The network comprises $44$ cells, seven of which are metered onramps, five mainline merging junctions, five mainline FIFO-diverging junctions and eight offramp FIFO-diverging junctions. In accordance with the assumptions, the mainline merging junctions upstream of $e_7, e_{14}, e_{20}, e_{28}$ and $e_{29}$ are controlled\footnote{It turns out that for the considered demand pattern, cells $e_{14}$, $e_{28}$ and $e_{29}$ do not congest in the optimal solution and hence, the junctions upstream of these cells do not require active flow control.} and all onramps $e_{38-44}$ are used for ramp metering. We also simulate the uncontrolled network for comparison, and in this case, we assume that both onramp and mainline merging junctions are described by the proportional-priority merging model. The turning rates for the mainline FIFO-diverging junctions are $\beta_{2,1} = 0.8$, $\beta_{31,1} = 0.2$, $\beta_{8,7} = 0.6$, $\beta_{21,7} = 0.2$, $\beta_{9,8} = 0.5$, $\beta_{35,8} = 0.5$, $\beta_{17,16} = 0.4$, $\beta_{33,16} = 0.6$, $\beta_{30,29} = 0.6$ and $\beta_{32,29} = 0.4$. At every offramp, $20\%$ of the flow leaves the network via the corresponding offramp, so e.g.\ $\beta_{4,3} = 1-0.2 = 0.8$. In the worst-case uncertainty realization, we assume that every lane is described by the same piecewise-affine fundamental diagram, with demand $\underline d( \rho(t) ) = \min \{ v \cdot \rho(t), F \}$ and supply $\underline s( \rho(t) ) = \min \{ F, ( \bar \rho - \rho(t) ) \cdot w \}$. The free-flow speed is $v = 120$km/h, the congestion-wave speed $w = 30$km/h, the lane capacity $F = 2000$cars/h and the traffic jam density $\bar \rho = 120$cars/km per lane. The number of lanes of each cell is displayed in Figure \ref{fig:topology}. 
Every cell is $l_e = 0.5$km long and the sampling time $\Delta t = 15$sec is chosen in accordance with Assumption \ref{assumption:fd}. External traffic demand arrives in the source cells. The maximal external demand $\overline w_{1}(t)$ arriving in cell $e_1$ is displayed in Figure \ref{fig:sample_both} (together with other uncertainty realizations). The maximal external demand arriving in cell $e_{10}$ is $\overline w_{10}(t) = \overline w_1(t)$, whereas $\overline w_{27}(t) = \overline w_{38-44}(t) = \frac{1}{2} \cdot \overline w_{1}(t)$. 

\begin{figure} 
	\centering
	\begin{subfigure}[b]{0.49\textwidth}
		\includegraphics[width = \textwidth]{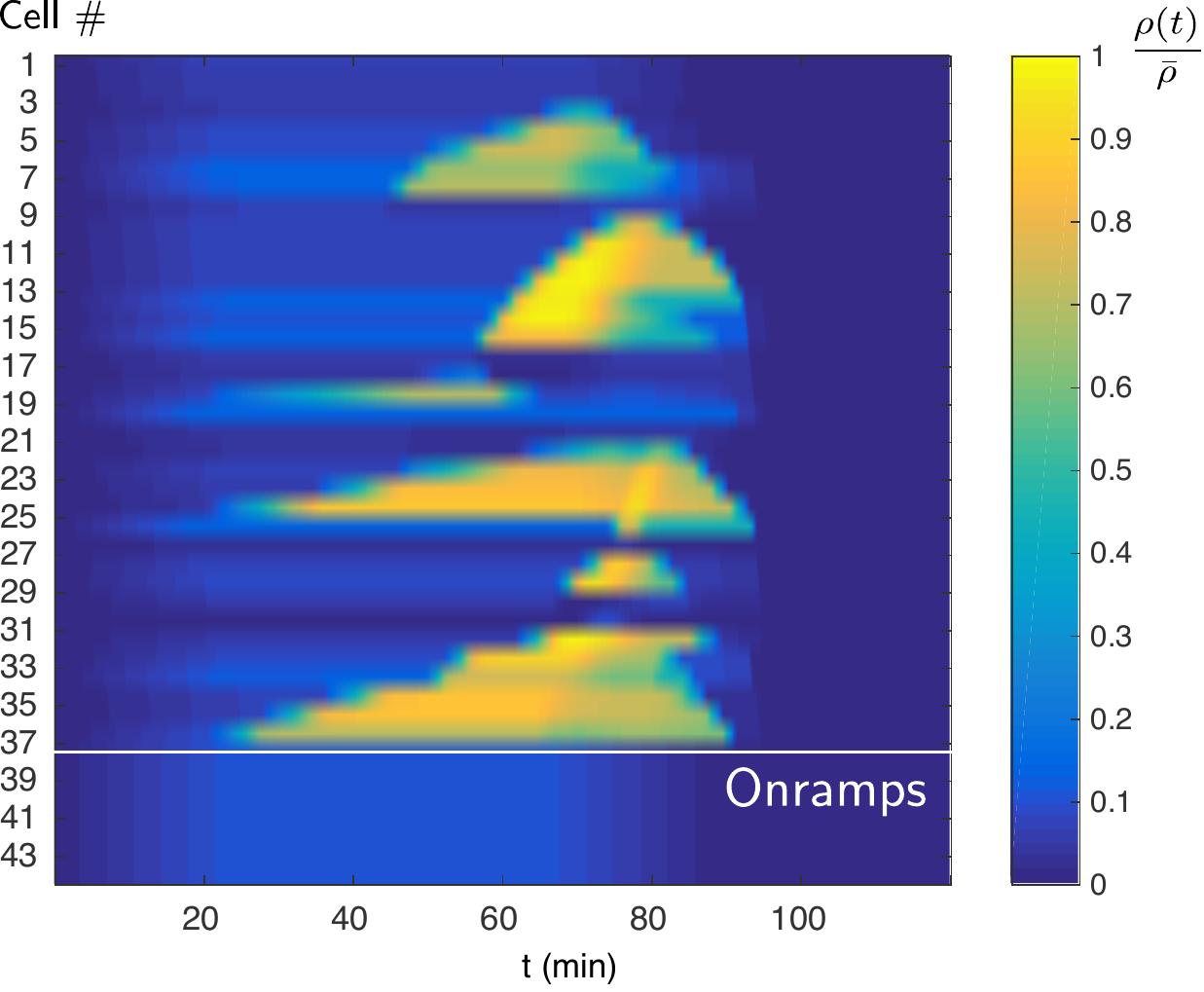}
		\caption{Evolution of the uncontrolled network for the worst-case uncertainty realization.}
		\label{fig:contour_openloop}
	\end{subfigure} ~ 
	\begin{subfigure}[b]{0.49\textwidth}
		\includegraphics[width = \textwidth]{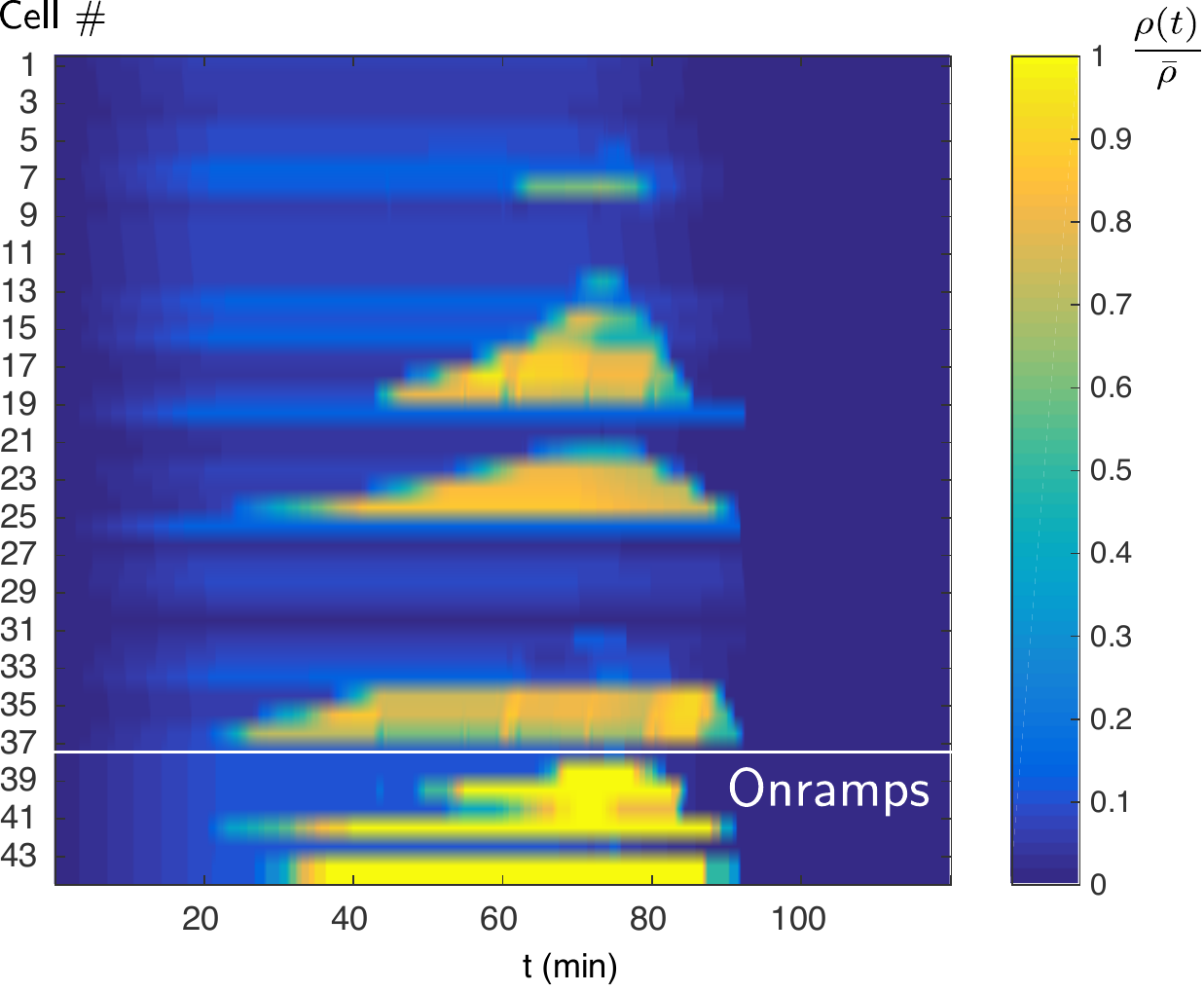}
		\caption{Optimal evolution of the controlled network for the worst-case uncertainty realization.}
		\label{fig:contour_optimal}
	\end{subfigure} 
\caption{Optimal control of merging flows achieves $\TTS^* = 590.0$h, improving over $\TTS_{\text{ol}} = 742.5$h for the uncontrolled network by keeping FIFO diverging junctions uncongested.}
	\label{fig:contour}
\end{figure}

The evolution of the uncontrolled network, for the worst-case uncertainty realization, is depicted in Figure \ref{fig:contour_openloop}. It turns out that cells $e_7$, $e_{16}$, $e_{20}$ and $e_{26}$ are bottlenecks which lead to the formation of congestion queues. In the uncontrolled, ``open-loop" case, $\TTS_{\text{ol}} = 742.5$h results. Control of merging flows allows for better performance. The optimal system evolution in a deterministic setting, where the uncertainty realization is equal to the worst-case uncertainty and known in advance, is depicted in Figure \ref{fig:contour_optimal}. Optimal control of merging flows achieves $\TTS^* = 590.0$h, which corresponds to a reduction of $\frac{\TTS_{\text{ol}} - \TTS^*}{\TTS_{\text{ol}}} = 20.5\%$ in comparison to the uncontrolled case. The contour plots reveal that this is achieved by ramp metering of cells $e_{39-42}$ and $e_{44}$ and prioritization of $\phi_{37}(t)$ over $\phi_{19}(t)$, which keeps cells $e_{4-8}$ (and cells $e_{25-22}$) uncongested and thereby increases flows through the FIFO-diverging junctions downstream of cells $e_{3}$, $e_{7}$, $e_{8}$, $e_{23}$ and $e_{24}$. In this example with piecewise-affine demand and supply functions, the optimal solution can be computed by solving a Linear Program (LP). The corresponding LP with 68684 variables and 137280 constraints is solved by Gurobi \cite{gurobi} in 40sec.\footnote{The solution was found using a 2013 MacBook Pro with 2.3GHz Intel i7 processor. Gurobi was interfaced via Matlab.} In the following, we will introduce uncertainty for external demands and fundamental diagrams and explore the impact on control performance.

\subsection{Verification of monotonicity} 

In this section, we aim to verify Theorem \ref{theorem:robust_counterpart} numerically. Instead of randomly sampling from the uncertainty sets, we use Corollary \ref{corollary:monotonicity_objective} and define a finite set of uncertainty realizations in a systematic manner, and compare the resulting TTS.
 
\begin{figure} 
	\centering
	\begin{minipage}{.31\textwidth}
		\begin{subfigure}[b]{\textwidth}
			\includegraphics[width = \textwidth]{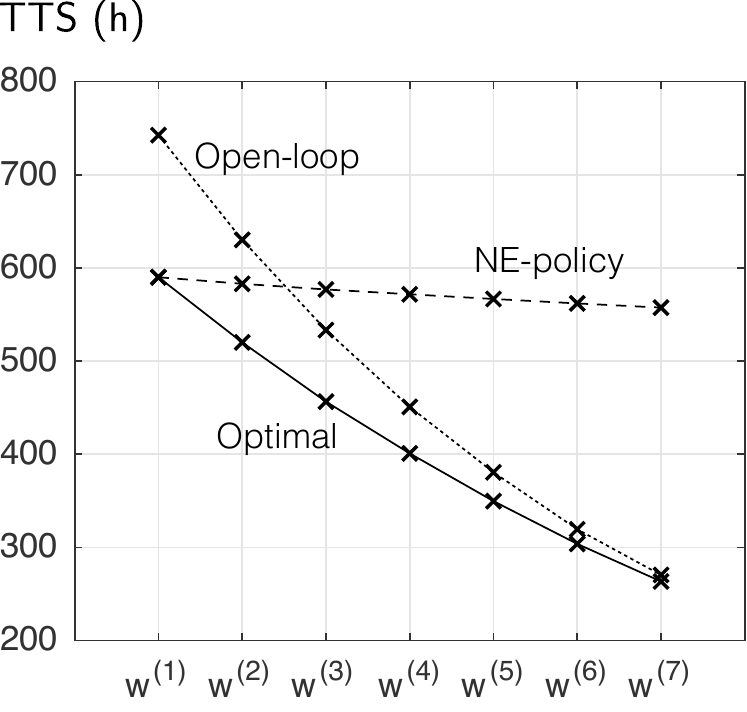}
			\caption{TTS for different realizations of the uncertain external demand, with $w_e^{(1)}(t) = \overline w_e(t)$ equal to the worst-case realization.}
			\label{fig:tts_externaldemand}
		\end{subfigure}
	\end{minipage} ~~
	\begin{minipage}{.31\textwidth}
		\begin{subfigure}[b]{\textwidth}
        			\includegraphics[width =\textwidth]{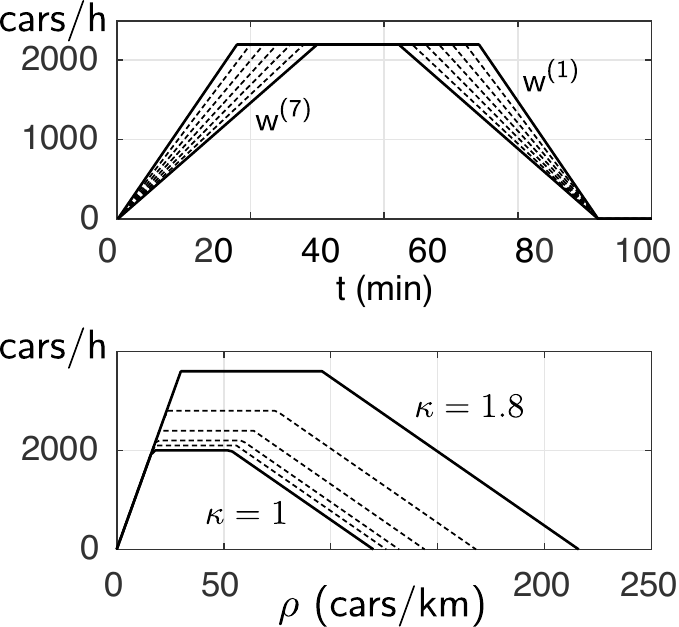}
       			\caption{Realizations of the external demand $w_1(t)$ (top) and of the fundamental diagram of a single lane (bottom).}
       			\label{fig:sample_both}
	       	\end{subfigure}
	\end{minipage} ~~
	\begin{minipage}{.31\textwidth}
		\begin{subfigure}[b]{\textwidth}
			\includegraphics[width = \textwidth]{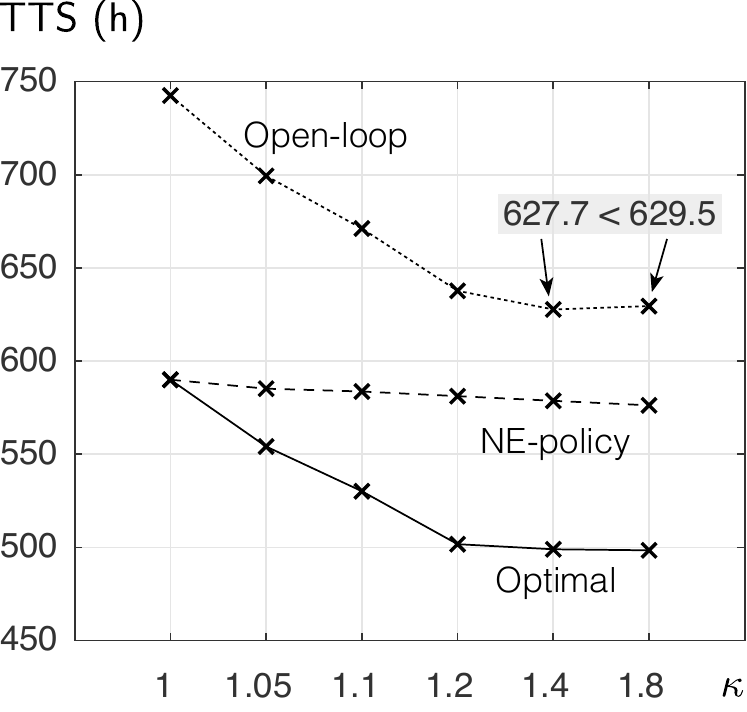}
			\caption{TTS for different realizations of the uncertain FD for cells $10-16$, $21-26$ and $33-34$.\\}
			\label{fig:tts_fd}
		\end{subfigure}
	\end{minipage}
%
\caption{For networks with controlled merging flows, the optimal TTS decreases if external demands decrease or demand and supply functions increase. The same holds for the NE policy, but it is not necessarily true for the uncontrolled network, as demonstated in Figure \ref{fig:tts_fd}.}
	\label{fig:correctness}
\end{figure} 

For the external demand, we define uncertainty realizations $w_e^{(k)}(t)$ for $k \in \{1,2,\dots,7\}$ as depicted in Figure \ref{fig:sample_both}. The realization $w_e^{(1)}(t) = \overline w_e(t)$ is the worst-case external demand and the remaining uncertainty realization are ordered in the sense that $w_e^{(1)}(t) \geq w_e^{(2)}(t) \dots \geq w_e^{(7)}(t)$, for all $e \in \E$ and all $0 \leq t \leq 2$h. This implies $w^{(i)} \in \W$. For now, we refrain from introducing uncertainty in the fundamental diagram and study the effect of the external demands on performance. We first solve the robust counterpart of the FNC problem with controlled merging junctions by employing Theorem \ref{theorem:robust_counterpart}, which reduces the robust control problem to a deterministic problem (in fact, the LP whose optimal solution is depicted in Figure \ref{fig:contour_optimal}). Thus we obtain an optimal reference trajectory $\rho^*_e(t)$ with flows $\phi^*_e(t)$ (or equivalently, a trajectory $z^*(t)$, $v^*(t)$ in the TCTM) for the worst-case uncertainty realization. Lemma \ref{lemma:tctm_monotonicity} implies that if the candidate policy $\phi_e(t) := \max \left\{ 0, \frac{z_e(t) - v_e^*(t)}{\Delta t} \right\}$, with $v_e^*(t) := z_e^*(t) - \Delta t \cdot \phi_e^*(t)$, is used, the evolution of the TCTM is monotone in the external demand. This means that the ordering of the trajectories 
 is preserved, for the ordered inputs $w^{(k)}$. Furthermore, the $\TTS$ decreases if the external demand decreases. We verify that trajectories are ordered numerically, but only depict the influence of the external demand on the $\TTS$ in Figure \ref{fig:tts_externaldemand}. It can be seen that for the NE policy, the $\TTS$ decreases as the external demand decreases, as predicted. For comparison, we also depict the $\TTS$ obtained for the uncontrolled system and the optimal $\TTS$, for uncertainty realizations known in advance. Corollary \ref{corollary:monotonicity_objective} states that the optimal $\TTS$ for networks with controlled merging junctions decreases if external demand decreases, and this is indeed the case in this example.\footnote{Note however, that the ordering of the trajectories is \emph{not} necessarily preserved under optimal control.} In this particular instance, the $\TTS$ obtained for the uncontrolled system also decreases if external demand decreases, however, this is not necessarily the case for other networks or demand patterns, as demonstrated earlier in Example \ref{example:ex1}. One important observation can be made from the comparison of the performance of the optimal solution, the uncontrolled network and the NE policy: The performance of the NE policy becomes vastly suboptimal if the actual external demand deviates significantly from the worst-case external demand. This suggests that the NE policy is far from subgame-perfect, as pointed out earlier. This is not a problem, however, as the main purpose of the NE policy is to aide in the proof of Theorem  \ref{theorem:robust_counterpart}, where it helps to ensure monotonicity of the TCTM. Better performance can be achieved by receding horizon strategies, as we will see in Section \ref{sec:mpc_study}.

Before exploring receding horizon control, we study the effect of uncertainty in the fundamental diagram. To this end, we assume that the external demand is equal to the worst-case external demand, and introduce different fundamental diagrams for cells $e_{10-16}$, $e_{21-26}$ and $e_{33-34}$. The capacity of those cells is increased in comparison to the worst-case fundamental diagram. This choice leaves the bottlenecks at cells $e_7$ and $e_{20}$ in place, but affects and potentially removes the bottlenecks at cells $e_{16}$ and $e_{26}$ (and prevents cells immediately upstream of the latter cells from becoming new bottlenecks). In particular, we consider a scaling factor $\kappa \geq 1$ and define demand and supply functions of the aforementioned cells as $d( \rho(t) ) = \min \{ v \cdot \rho(t), \kappa F \}$ and supply $s( \rho(t) ) = \min \{ \kappa F, (\kappa \bar \rho  - \rho(t)) \cdot w \}$. Demand and supply functions of other cells remain unchanged. The fundamental diagram of a single lane, for different values of $\kappa$, is depicted in Figure \ref{fig:sample_both}. Again, we compare the performance of the uncontrolled system, the NE policy and the optimal solution (assuming perfect knowledge of all uncertainty realizations in advance) for different values of $\kappa$. It is important to note that every demand function considered in this example is non-decreasing and every supply function is non-increasing. Hence, they all satisfy Assumption \ref{assumption:fd} and for any $\kappa \geq 1$, the corresponding TCTM is monotone in $z(t)$. In turn, this implies that $\TTS$ decreases as $\kappa$ increases, for both the NE policy and the optimal solution. The results are depicted in Figure \ref{fig:tts_fd} and confirm these predictions. For comparison, we also plot the $\TTS$ achieved in the uncontrolled system. It turns out that $\TTS$ in the uncontrolled network \emph{increases} if $\kappa$ is increased from $1.4$ to $1.8$, demonstrating again that Corollary \ref{corollary:monotonicity_objective} does not extend to networks with uncontrolled merging junctions.

While the numerical studies performed so far confirm the theoretical results, they also reveal that even though the worst-case performance bound is satisfied, the NE policy can perform much worse than the optimal solution, or even the uncontrolled network, if uncertainty realizations other than the worst-case are encountered. This extreme conservativeness makes the NE policy impractical. 
In the next section, we will therefore study the performance of receding horizon policies that retain the worst-case performance guarantees of the NE policy, but approximate the optimal solution more closely.

\subsection{Performance of receding-horizon control} 
\label{sec:mpc_study}

In this section, we aim to analyze the performance of receding horizon policies for the robust FNC problem. For the most part, we consider policies with a terminal cost as described in Proposition \ref{proposition:mpc}, for which the worst-case performance guarantees hold. For such a policy, the control horizon $T^c$ needs to be chosen. In general, this choice is a trade-off between computational effort required to solve the optimization problem at every time step and performance, which tends to improve for longer horizons (altough there is no guarantee that performance will improve if a longer horizon is chosen). The robust FNC \eqref{eq:robust_counterpart} itself is defined as a finite horizon problem, with horizon length $T$. For a control horizon length $T^c = T$, we obtain the subgame-perfect policy described at the end of Section \ref{sec:main_result} (recall that the control horizon at time $t$ is truncated if $t + T^c > T$). In addition to choosing the control horizon, one can also choose to not re-optimize after every sampling interval $\Delta t$. Instead, one might choose to only re-optimize every $k \in \Z^+$ time steps.\footnote{If one choses to do so, then the control horizon has to be an integer-multiple of the time between re-optimization, to ensure that the worst-case performance guarantee from Proposition \ref{proposition:mpc} continues to hold. The intuition is that in this case, one could define a new system, with sampling time $k \cdot \Delta t$, and perform the analysis of the receding horizon controller for this new system.}

\begin{figure} 
	\centering

	\begin{subfigure}[b]{0.32\textwidth}
		\includegraphics[width = \textwidth]{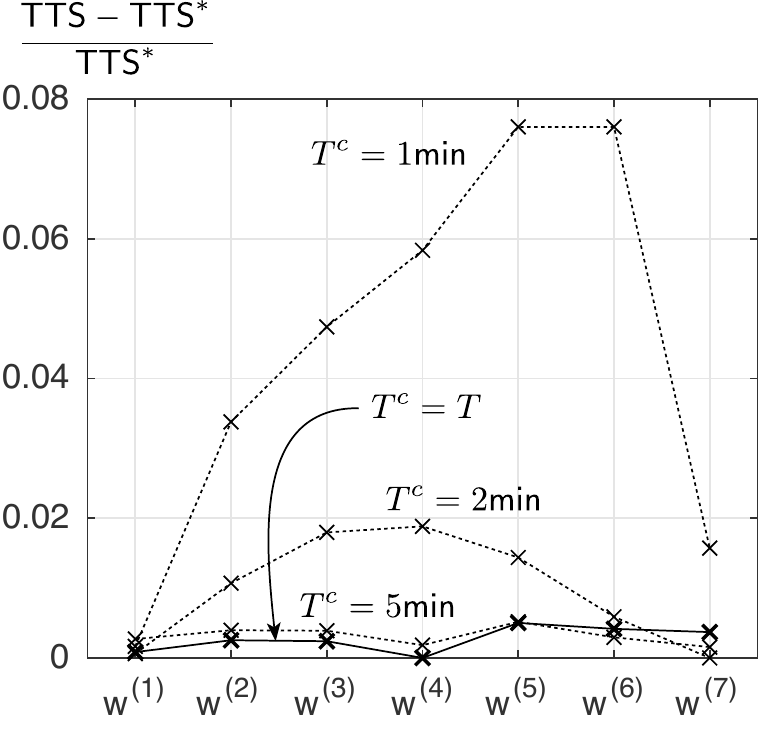}
		\caption{Suboptimality as a function of the external demand, for different horizon lengths.}
		\label{fig:horizon_externaldemand}
	\end{subfigure} ~ 
	\begin{subfigure}[b]{0.32\textwidth}
		\includegraphics[width = \textwidth]{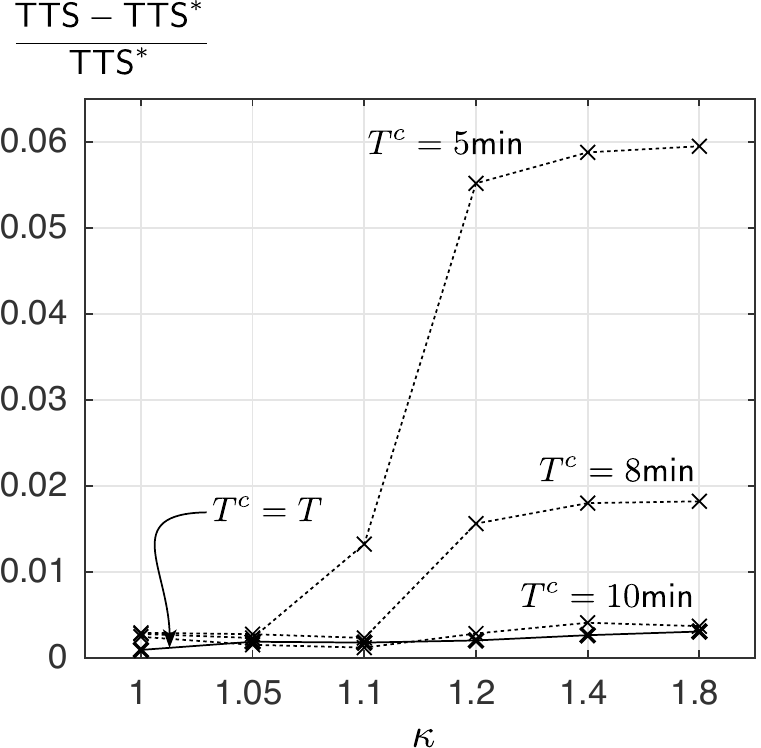}
		\caption{Suboptimality as a function of the scaling factor $\kappa$, for different horizon lengths.}
		\label{fig:horizon_fd}
	\end{subfigure} ~
	\begin{subfigure}[b]{0.32\textwidth}
		\includegraphics[width = \textwidth]{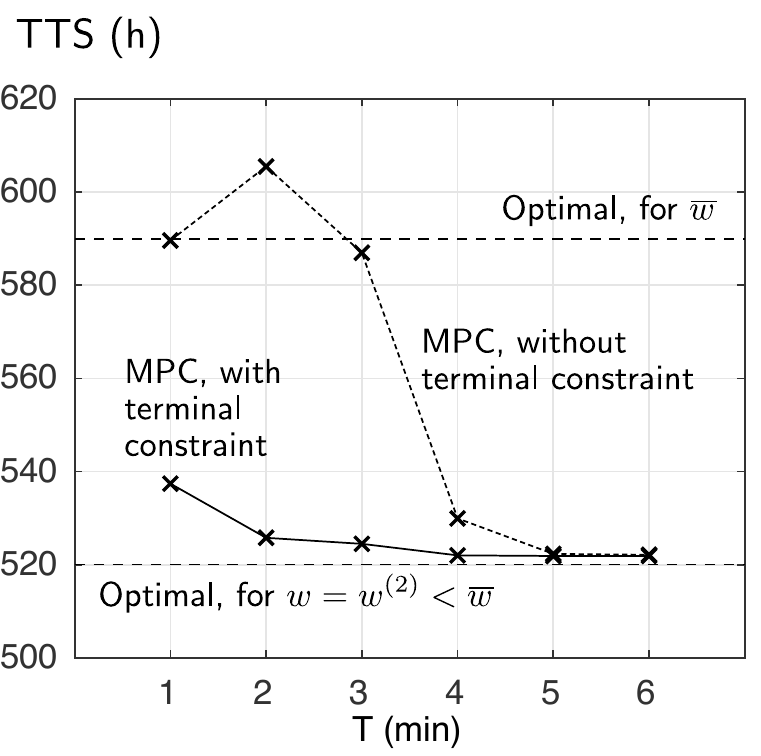}
		\caption{TTS as a function of horizon length of the receding horizon controller, for the special case $w^{(2)}$.}
		\label{fig:tts_horizon}
	\end{subfigure} 
\caption{Suboptimality tends to decrease with increasing horizon length of the receding horizon controller, up to a point. For uncertain external demand, the performance for $T^c \geq 5$min is similar to the subgame-perfect solution ($T^c = T$), see Figure \ref{fig:horizon_externaldemand}. For uncertain fundamental diagrams, this is the case for $T^c \geq 10$min, see Figure \ref{fig:horizon_fd}. In both cases, a receding horizon policy with terminal constraint is used. A policy without terminal constraint may increase TTS over the optimal TTS for the worst-case uncertainty realization (Figure \ref{fig:tts_horizon}), while policies with terminal constraint are guaranteed to be improving.}
	\label{fig:horizon}
\end{figure}

In this part of the numerical study, we evaluate the performance of receding horizon policies with different control horizon lengths for the robust FNC problem, for the network and the uncertainty realizations studied in the previous section. In all cases, we choose to re-optimize every $k =4$ time steps. For $\Delta t = 15$sec, this implies that we re-optimize every $1$min. Note that using a sampling time $\Delta t \leq 15$sec is required for this particular traffic network, due to the stability constraint encountered when discretizing the traffic model, according to Assumption \ref{assumption:fd}. In addition, we assume that at time $t$, the uncertainty realizations $\omega_\tau$ for the time until the next re-optimization, that is, for $t \leq \tau < t + 60$sec, are known without uncertainty. This additional assumption helps to prevent unrealistic underutilization of the network. For example, if a conservative estimate for the capacity of a merging junction is used even for the near-future (in this case, the next $60$sec), then the controlled flows will always be restricted to serving a total flow less or equal to this worst-case capacity. This behavior is unrealistic, however, since in practice, flow controls would  be translated into traffic light duty-cycles before being implemented. If local traffic happens to flow faster than in worst-case predictions, then larger flows are automatically being served during green times. Also, it is realistic to assume that short-term predictions are much less uncertain than long term predictions.

The results for different realizations of the external demand are depicted in Figure \ref{fig:horizon_externaldemand}. This figure uses the same horizontal axis as Figure \ref{fig:tts_externaldemand}, but instead of displaying the TTS achieved by the different policies, the relative suboptimality $\frac{\TTS - \TTS^*}{\TTS^*}$ is shown. As in the previous section, the optimal solution refers to the solution of the deterministic FNC problem with perfect anticipation of the values of the ``uncertain" quantities. In this particular instance, even short control horizons tend to perform well. In particular, for $T^c \geq 5$min, the performance of the receding horizon policy is similar to the one of the subgame-perfect policy ($T^c = T$). The performance deterioration in comparison to the optimal solution with perfect information is less than $0.5\%$ for such horizon lengths. Similar results are obtained for uncertain demand and supply functions. Figure \ref{fig:horizon_fd} reveals that somewhat longer control horizons are required to achieve comparable performance, but for $T^c \geq 10$min, the performance of the receding horizon policy is again similar to the subgame-perfect policy and a performance deterioration of less than $0.5\%$ in comparison to the optimal solution is achieved. 

However, the receding horizon policies with short horizons and the subgame-perfect policy differ in terms of their computational effort. In particular, solving optimization problems with horizon $10$min takes $0.174$sec on average and at most $0.422$sec\footnote{Solver time reported by Gurobi, again on a 2013 MacBook Pro with 2.3GHz Intel i7 processor.}. In terms of computational effort, real-time implementation of this policy seems to be unproblematic if re-optimization every minute is required. By contrast, the subgame-perfect policy requires to solve optimization problems with horizon lengths up to $T$. In this numerical study, solving such problems takes up to $56.2$sec. It is questionable if such a policy is suitable for real-time implementation, in particular if the total horizon of interest $T$ exceeds $2$h, if larger networks are considered or if additional tasks like state estimation have to be performed.

Finally, we also study the effect of the terminal cost on performance. To this end, we consider only the case $w = w^{(2)}$ and depict the performance with and without terminal constraint in Figure \ref{fig:tts_horizon}, for different control horizon lengths. The policy with terminal cost satisfies the worst-case performance bound for any horizon length, as expected, but it turns out that the policy without terminal constraint performs \emph{worse} than the performance bound computed using the worst-case uncertainty, for horizon length $2$min. This example demonstrates that the performance guarantees of Proposition \ref{proposition:mpc} do not extend to receding horizon policies without terminal constraint. It appears that if the control horizon is chosen long enough, in this case, $T^c \geq 5$min, the terminal cost has negligible influence on performance. This is expected for very long horizons, in particular if $T^c \approx T$, but it is not clear a priori how long exactly the control horizon needs to be chosen in order to safely neglect the terminal constraint.

\section{Conclusions} \label{sec:conclusions} 

In this work, we have considered the robust FNC problem, in which external traffic demands and demand and supply functions are uncertain. We have shown that if merging junctions are controlled and uncertainty sets are given as point-wise bounds on the uncertain quantities, then the robust FNC problem is equivalent to a convex, finite-dimensional optimization problem. This result is based on the insight that for networks with controlled merging junctions, the worst-case uncertainty realization is easy to identify since it is equal to the upper bound on the external demand, respectively the lower bound on demand and supply functions. We have also demonstrated via counterexamples that the same is not necessarily true for networks with uncontrolled merging junctions. The main result, the reformulation of the robust FNC problem as a convex problem, allowed for the design of computationally efficient receding horizon policies, with worst-case performance guarantees. In the numerical study, these policies showed promising performance, also for uncertainty realizations other than the worst-case.

The main contribution of this work is theoretical, with the objective to identify a robust counterpart of the FNC whose solution is tractable, i.e., that can be reduced to a convex optimization problem. Several assumptions have been necessary to do so, in particular, we have assumed that the uncertainty set of the external demand is defined as point-wise, upper bounds. Future work might seek to generalize this assumption. In particular, allowing for general, polyhedral uncertainty sets for the external demands has the potential to allow for less conservative uncertainty sets. For practical application, it will also be necessary to construct uncertainty sets from recorded data. Since the existence of outliers is expected in all sufficiently large, real-world data sets, one might want to use less conservative uncertainty sets, such that individual, future uncertainty realizations lie within these sets with high, but not certain probability. This raises the question if one can combine data-driven construction of these uncertainty sets with controller design using techniques from stochastic optimization -- instead of robust optimization, as in this work -- to obtain probabilistic performance guarantees.

\FloatBarrier
\bibliographystyle{plain}
\bibliography{/Users/mschmitt/Documents/Docs_Latex/MS_bibliography}

\FloatBarrier

\begin{appendix}

\section{Proof of Lemma } \label{appendix:relaxation2} 

\begin{proof} 
Note that problem \eqref{eq:mpc} is the convex relaxation of the optimization problem 
\begin{equation} 
\label{eq:mpc_ap}
\begin{array}{rrll}
 & \underset{\phi( \tau ), \rho( \tau )}{\text{minimize}} & \sum_{ \tau = t+1 }^{\max \{ t+T_c, T \}} l^\top \rho(\tau)  \\ [2ex]
 & \text{subject to} & \text{FNC dynamics \eqref{eq:ctm_densities}, \eqref{eq:ctm_flows} and constraints \eqref{eq:ramp_constraints}, \eqref{eq:ctm_constraints} for $\omega = \overline \omega$.}  \\[1ex]
 & & P L \cdot \big( \rho(t+T_c) - \rho^*(t+T_c) \big) \leq 0  \\[1ex]
 & & \text{Initial state $\rho( t )$ given.}
\end{array}
\end{equation}
If the latter optimization problem is expressed in terms of the TCTM variables, one obtains problem \eqref{eq:mpc_nonconvex}. Due to equivalence of the alternative system representations according to Lemma \ref{lemma:equivalence_tctm}, it is sufficient to show that for any solution to \eqref{eq:mpc}, one can find a solution to \eqref{eq:mpc_ap} such that the controlled flows $\phi_e(t)$ for $e \in \N$ coincide.

For the FNC problem with controlled merging junctions, but without terminal constraints, \cite[Theorem 2]{schmitt2017exact} ensures just that. The proof is based on the existence of a concave, state-monotone reformulation. The theorem extends to systems with additional constraints, such as the terminal constraints at hand, as long as the new constraints are also concave and state-monotone in the reformulated system (see \cite[Corollary 1]{schmitt2017exact}). In particular, consider the \emph{cumulative flows} $\Phi_e(t) := \Delta t \cdot \sum_{\tau = 0}^{T-1} \phi_e(t)$ for $e \in \E$ and the auxiliary states $\hat \Phi_e(t) := -\Phi_e(t)$ for $e \in \N$, as defined in \cite{schmitt2017exact}. Furthermore, we define $\overline W(t) := \sum_{\tau = 0}^{t-1} \overline w(\tau)$. The terminal constraints can be expressed as
\begin{align*}
P L &\cdot \big( \rho^*(t+T_c) - \rho(t+T_c) \big) \\
 &= P L \cdot \rho^*(t+T_c) - p^\top_e L \rho(t) - \Delta t \cdot \sum_{\tau = t}^{t+T_c-1} \bigg( p_e^\top(t) \overline w(\tau) + \sum_{i \in \E} p_e(i) \Big( {\sum}_{j \in \N_S} \beta_{i,j}\phi_j(\tau) \Big) - \phi_e(\tau) \bigg) \\
 &= P L \cdot \rho^*(t+T_c) - p^\top_e L \rho(t) + \Delta t \cdot \bigg( - p_e^\top(t) \big( \overline W(t+T_c)- \overline W(t) \big) \dots \\
 &\hspace{1cm} + \sum_{i \in \E} p_e(i) \cdot \Big( {\sum}_{j \in \N_S} \beta_{i,j} \big( \hat\Phi_j(t+T_c) - \hat\Phi_j(t) \big) \Big) + \big( \Phi_e(t+T_c) - \Phi_e(t) \big) \bigg) \geq 0.
\end{align*}
In the final inequality, $\rho^*(t+T_c)$, $\rho(t)$, $\hat \Phi_e(t)$ (for $e \in \N$), $\Phi_e(t)$, $\overline W(t)$ and $\overline W(t + T_c)$ are known parameters in the optimization problem at time $t$. Only the quantities $\Phi(t + T_c)$ and $\hat \Phi_e(t+T_c)$ (for $e \in \N$) are optimization variables. The right-hand side (RHS) of the final inequality is monotone\footnote{The reformulation in terms of cumulative flows \cite{schmitt2017exact} poses the FNC as a \emph{maximization} problem. For this reason, we need to ensure that the RHS of constraints of the form $g_t( \Phi(t), \hat \Phi(t) ) \geq 0$ is monotone. This is in contrast to Definition \ref{definition:monotone_system}, where we require constraints of the form $g_t( \Phi(t), \hat \Phi(t) ) \leq 0$ (note the order of the inequality!) with monotone RHS. } in these optimization variables, and therefore, a generalization of \cite[Theorem 2]{schmitt2017exact} applies and problems \eqref{eq:mpc} and \eqref{eq:mpc_ap}, and in turn, problem \eqref{eq:mpc_nonconvex}, define identical receding horizon policies.
\end{proof} 

\end{appendix}

\end{document}